\documentclass[a4paper, 10pt]{amsart}
\usepackage{amsthm}
\usepackage[]{amsmath}
\usepackage{amssymb}
\usepackage{enumerate}
\usepackage{tabularx}
\usepackage[]{color}
\usepackage[left=3.5cm, right=3.5cm, top=3cm, bottom=3cm]{geometry}
\usepackage[colorlinks]{hyperref}
\usepackage{tikz}
\usepackage{multirow}
\usepackage{diagbox}
\usepackage{subcaption}
\usepackage[ruled, vlined]{algorithm2e}
\usepackage{xcolor}

\usetikzlibrary{calc}

\allowdisplaybreaks[4]

\title[Unfitted FEM with low regularity estimates]{An unfitted finite
element method for elliptic interface problem with low regularity
estimates}

\author[F.-Y. Yang]{Fanyi Yang} \address{School of Mathematics
, Sichuan University, Chengdu 610065, P.R. China}
\email{yangfanyi@scu.edu.cn}

\newcommand{\bm}[1]{\boldsymbol{#1}}
\newcommand{\bmr}[1]{\bm{\mr{#1}}}
\newcommand{\lj}{[ \hspace{-2pt} [}
\newcommand{\rj}{] \hspace{-2pt} ]}
\newcommand{\mb}[1]{\mathbb{#1}}
\newcommand{\mc}[1]{\mathcal{#1}}
\newcommand{\mr}[1]{\mathrm{#1}}
\newcommand{\jump}[1]{\lj #1 \rj}
\newcommand{\aver}[1]{ \{#1\}  }

\newcommand{\dbx}[1]{\mathrm d \boldsymbol{#1}}
\newcommand{\dx}[1]{\mathrm d {#1}}

\newcommand{\Bw}[1]{B_{\omega(#1)}}
\newcommand{\Bt}[1]{B_{\tau(#1)}}

\def\mS{\mc{S}}
\def\mN{\mc{N}}

\def\bu{\bm{\nu}}
\def\bc{\bmr{c}}

\def\MTh{\mc{T}_h}
\def\MThi{\mc{T}_{h, i}}
\def\MThic{\mc{T}_{h, i}^{\circ}}
\def\MThoc{\mc{T}_{h, 0}^{\circ}}
\def\MThlc{\mc{T}_{h, 1}^{\circ}}
\def\MThi{\mc{T}_{h, i}}
\def\MThl{\mc{T}_{h, 1}}
\def\MTho{\mc{T}_{h, 0}}

\def\MFhi{\mc{F}_{h, i}}
\def\MFhic{\mc{F}_{h, i}^{\circ}}

\def\MFhi{\mc{F}_{h, i}}
\def\MFhl{\mc{F}_{h, 1}}
\def\MFho{\mc{F}_{h, 0}}

\def\mE{\mc{E}}
\def\MThG{\mc{T}_h^{\Gamma}}

\def\Ohi{\Omega_{h, i}}
\def\Ohic{\Omega_{h, i}^{\circ}}
\def\Oho{\Omega_{h, 0}}
\def\Ohl{\Omega_{h, 1}}
\def\OhG{\Omega_h^{\Gamma}}
\newcommand{\DGenorm}[1]{ \| #1\|_{e}}
\newcommand{\shnorm}[1]{| #1 |_{s_h}}
\newcommand{\ghnorm}[1]{| #1 |_{g_h}}
\newcommand{\Vnorm}[1]{|#1|_{V}}
\newcommand{\Vhnorm}[1]{|#1|_{V_h}}
\newcommand{\ghonorm}[1]{| #1 |_{g_{h, 0}}}
\newcommand{\ghlnorm}[1]{| #1 |_{g_{h, 1}}}
\newcommand{\ghinorm}[1]{| #1 |_{g_{h, i}}}
\newcommand{\tRomannum}[1]{\text{(\uppercase\expandafter{\romannumeral #1})}}
\newcommand{\sgDnorm}[1]{|#1|_{g_D, \Gamma}}
\newcommand{\tu}[1]{\textbf{\textup{#1}}}

\def\un{\bmr{n}}

\newtheorem{assumption}{Assumption}
\newtheorem{theorem}{Theorem}
\newtheorem{lemma}{Lemma}
\newtheorem{remark}{Remark}

\begin{document}

\maketitle

\begin{abstract}
  In this paper, we present and analyze an unfitted finite element
  method for the elliptic interface problem. 
  We consider the case that the interface is $C^2$-smooth or
  polygonal, and the exact solution $u \in H^{1+s}(\Omega_0 \cup
  \Omega_1)$ for any $s > 0$.
  The stability near the interface is guaranteed by a local
  polynomial extension technique combined with ghost penalty
  bilinear forms, from which the robust condition number estimates and
  the error estimates are derived.
  Furthermore, the jump penalty term for weakly enforcing the jump
  condition in our method is also defined
  based on the local polynomial extension, which enables us to
  establish the error estimation particularly for solutions with low
  regularity.
  We perform a series of numerical tests in two and three dimensions
  to illustrate the accuracy of the proposed method.

\noindent \textbf{keywords}: 
unfitted mesh; 
low regularity estimates; 
polynomial extension;

\end{abstract}


\section{Introduction}
\label{sec_introduction}
The interface problems are commonly encountered in scientific and
engineering applications, where the governing partial differential
equations are typically coupled through interfaces.
The finite element method has long served as a standard computational
tool for solving such problems, and 
over the past decades, various finite element methods have been
developed for solving the elliptic interface problem. 
The finite element method can be roughly classified into fitted and
unfitted methods based on the mesh type. 
The body-fitted mesh is required to be aligned with the interface.
For complex geometries, it is a very time-consuming task to generate 
a fitted and high quality mesh \cite{Huang2017unfitted,
Liu2020interface}. 
In recent years, the unfitted finite element method has emerged as a
very appealing approach, as the interface description is decoupled
from the mesh generation, which provides a great flexibility for
handling complex geometries.
Examples of unfitted finite element methods are the cut finite element
method \cite{Hansbo2002unfittedFEM, Burman2015cutfem,
Massing2019stabilized, Bordas2017geometrically, Burman2021unfitted,
Liu2020interface, Burman2025cut, Li2020interface, Huang2017unfitted,
Wu2012unfitted} and 
the immersed finite element method \cite{Li1998immersed,
Li2006immersed, Lin2019nonconforming, Guo2019higher}.

In \cite{Hansbo2002unfittedFEM}, the authors proposed an unfitted
finite element
method for solving the elliptic interface problem, wherein the
numerical solution is constructed from two distinct finite element
spaces, and the jump conditions are weakly imposed by Nitsche's
penalty method.
This idea has been a popular discretization for interface problems,
now known as the cut finite element method (CutFEM), and has been
applied to many interface problems \cite{Hansbo2014cut,
Guzman2018infsup, Burman2012ficticious, Burman2014fictitious,
Liu2020interface, Li2023curl, Yang2024least}. 
We refer to the survey papers \cite{Burman2015cutfem, Burman2025cut,
Gurkan2019stabilized} and the references therein for further advances.
For such penalty methods, the small cuts near the interface have to be
treated very carefully, which may adversely affect the conditioning of
the method and even hamper the convergence \cite{Burman2021unfitted}.  
Some control from the physical domain to the entire active mesh is
generally necessary \cite{Burman2025cut}.
A widely adopted approach is to add suitable stabilization terms to
the weak form, 
such as the ghost penalty form \cite{Burman2010ghost,
Gurkan2019stabilized}, which serves to extend the control of the
relevant norms across the physical domain and the active mesh.
In addition to this method, the strategies based on the agglomeration of
elements and the extension techniques 
have also been employed for the stabilization, see
\cite{Johansson2013high, Badia2018aggregated, Neiva2021robust,
Huang2017unfitted, Chen2021an, Burman2021unfitted, Burman2022cutfem,
Yang2022an} for some examples. 
The main idea is to extend the numerical solution from the interior 
stable elements to cut elements for stabilization.

To our best knowledge, almost all of the existing works establish the
error estimation under the assumption that the exact solution has at
least piecewise $H^2$ regularity and the interface is $C^2$.
For the case where the interface is only polygonal, the exact solution
may be of low regularity \cite{Costabel1999singularities}, and there
are few studies concerning such cases.
We note that the $H^1$ trace estimate is typically the main tool to
estimate the approximation error on the interface.
Applying this trace estimate to the gradient of the exact solution
requires the solution is at least $H^2$. 
For the case of low regularity, this tool is no longer applicable.
In a recent work \cite{Burman2024low}, the authors proposed a CutFEM
approximation for the elliptic problem with mixed boundary conditions
under minimal regularity assumptions, i.e. the exact solution $u \in
H^s$ with $1 \leq s < 3/2$.
They introduced a regularized finite element formulation and designed
a cutoff function to address the difficulty that the solution is not
regular enough.

In this paper, we present an unfitted finite element method for the
elliptic interface problem under the assumption that the solution is
piecewise $H^{1+s}$ with $s > 0$, following the CutFEM framework.
In our method, we adopt the local polynomial extension technique to
give a suitable penalty bilinear to cure the effects brought by small
cuts. 
More importantly, the jump penalty term, which is used to weakly
impose the inhomogeneous jump conditions, is also constructed from the
local polynomial extension. 
By this jump term, we establish an inverse-type estimate on the
interface and further construct a linear mapping from the extended
finite element space to the $H^1$ space. 
We then follow the framework outlined in \cite{Gudi2010new} to prove
the optimal error estimates under the $H^1$ norm for low regularity
solutions based on that constructed linear mapping.
These theoretical results are confirmed by a series of numerical
experiments.

The rest of this paper is organized as follows. 
In Section \ref{sec_preliminaries}, we introduce the elliptic
interface problem and give the notation used in the numerical method. 
The properties of the direct polynomial extension operator are also
included.
In Section \ref{sec_scheme}, we develop the numerical scheme and
present the error estimation.
Finally, Section \ref{sec_numericalresults} reports 
a series of numerical tests in both two and three dimensions
to illustrate the numerical performance for the proposed method.
In Appendix \ref{sec_app}, we present a property of the barycenter of
a given simplex.


\section{Problem setting and preliminaries}
\label{sec_preliminaries}
Let $\Omega \subset \mb{R}^d(d = 2, 3)$ be an open polygonal
(or polyhedral) domain.
Let $\Gamma \Subset \Omega$ be a $C^2$-smooth or polygonal
(polyhedral) interface, which separates $\Omega$ into two disjoint
subdomains $\Omega_0$ and $\Omega_1$. 
Here, $\Omega_0 \cap \Omega_1 = \varnothing$ and $\overline{\Omega}_0
\cup \overline{\Omega}_1 = \overline{\Omega}$.
In this paper, we study the elliptic interface problem proposed on
$\Omega_0 \cup \Omega_1$, which reads: find $u \in H^1(\Omega_0 \cup
\Omega_1)$ such that 
\begin{equation}
  \begin{aligned}
    -\nabla \cdot (\alpha \nabla u) &= f, && \text{in } \Omega_0 \cup
    \Omega_1, \\
    u &= 0, && \text{on } \partial \Omega, \\
    \jump{u} = g_D \un_{\Gamma}, \quad \jump{\alpha \nabla_{\un} u} &
    = g_N, &&
    \text{on } \Gamma, \\
  \end{aligned}
  \label{eq_problem}
\end{equation}
where $\alpha$ is a positive piecewise constant function on $\Omega_0
\cup \Omega_1$ and $\jump{\cdot}$ denotes the jump of a function on
the interface (see \eqref{eq_jump}). 
The weak formulation to the problem \eqref{eq_problem} is to find $u
\in H^1(\Omega_0 \cup \Omega_1)$ with $\jump{u}|_{\Gamma} = g_D$ and
$u|_{\partial \Omega} = 0$ such that
\begin{displaymath}
  a(u, v) = l(v), \quad \forall v \in H_0^1(\Omega), 
\end{displaymath}
where 
\begin{equation}
  a(w, v) := \int_{\Omega_0 \cup \Omega_1} \alpha \nabla w \cdot
  \nabla v \dbx{x}, \quad l(v) := \int_{\Omega_0 \cup \Omega_1} f v
  \dbx{x} + \int_{\Gamma} g_N v \dbx{s}.
  \label{eq_weakform}
\end{equation}
In our study, we assume that for a fixed $s > 0$, the data functions
$f \in H^{s-1}(\Omega_0 \cup \Omega_1), g_D \in H^{s-1/2}(\Gamma), g_N
\in H^{s-3/2}(\Gamma)$ and the interface problem \eqref{eq_problem}
admits a unique solution $u \in H^{1+s}(\Omega_0 \cup \Omega_1)$.
We refer to \cite{Kellogg1975poisson, Kellogg1972higher} for more
details on regularity results to \eqref{eq_problem}.
We assume that the exact solution $u^i := u|_{\Omega_i}$ can be
extended to the entire domain such that $u^i \in H^{1+s}(\Omega)$ and
$u|_{\Omega_i} = u^i|_{\Omega_i}$ with $\| u^i \|_{H^{1+s}(\Omega)}
\leq C \| u \|_{H^{1+s}(\Omega_i)}$ for both $i = 0, 1$.

Let us introduce the notation required in the discrete numerical
scheme.
We denote by $\MTh$ a quasi-uniform partition of $\Omega$ into
triangular (tetrahedral) elements. 
The mesh is unfitted, meaning that the element faces in the mesh
are not required to align with the interface $\Gamma$. 
Let $h_K$ denote the diameter of $K \in \MTh$, and let $\rho_K$ denote
the radius of the largest ball inscribed in $K$.
The mesh size $h$ is given as 
$h := \max_{K \in \MTh} h_K$. 
Let $\rho := \min_{K \in \MTh} \rho_K$, and 
the mesh is quasi-uniform in the sense that there exists a constant
$C_{\nu}$ independent of $h$ such that $h \leq C_{\nu} \rho$.

We next give the notation related to subdomains and the interface. 
For $i = 0, 1$, we define 
\begin{align*}
  \MThi := \{K \in \MTh: \ K \cap \Omega_i \neq \varnothing\}, \quad
  \MThic  := \{K \in \MThi: \ K \subset \Omega_i\},
\end{align*}
Here, $\MThi \subset \MTh$ is the minimal subset of elements that
fully covers $\Omega_i$, i.e., the active mesh of $\Omega_i$, 
and $\MThic$ consists of all elements entirely contained in
$\Omega_i$. 
We define $\MThG :=  \{K \in \MTh: \ K \cap \Gamma \neq \varnothing\}$
as the set of all cut elements. 
The associated domains are defined as
\begin{displaymath}
  \Ohi := \text{Int}(\bigcup_{K \in \MThi} \overline{K}), \quad \Ohic
  := \text{Int}(\bigcup_{K \in \MThic} \overline{K}), \quad \OhG :=
  \text{Int}(\bigcup_{K \in \MThG} \overline{K}),
\end{displaymath}
and it is clear that $\Ohic \subset \Omega_i \subset \Ohi$.
For any cut $K \in \MThG$, we let $\Gamma_K := K \cap \Gamma$. 
For the mesh $\MThi$, we let $\MFhic$ be the set of all interior $d-1$
dimensional faces in $\Ohi$. 
For any $f \in \MFhic$, we let $h_f$ denote its diameter and we let
$f^i := f \cap \Omega_i$.
For domains $\Oho$ and $\Ohl$, the following $C^0$ finite element
spaces 
\begin{align*}
  V_{h, 0}^m & := \{v_h \in H^1(\Oho): v_h|_K \in \mb{P}_m(K), \quad
  \forall K \in \MTho \}, \\
  V_{h, 1}^m & := \{v_h \in H^1(\Ohl): v_h|_{\partial \Omega} = 0,
  \quad v_h|_K \in \mb{P}_m(K), \quad \forall K \in \MThl\}, 
\end{align*}
will be employed in the scheme to approximate the exact solutions
$u^0, u^1$ over $\Omega_0$ and $\Omega_1$, respectively.
We set $V_h^m := V_{h, 0}^m \cdot \chi_0 + V_{h, 1}^m \cdot \chi_1$,
where $\chi_i$ denotes the characteristic function of $\Omega_i$. 
It is evident that any $v_h \in V_h^m$ admits a unique decomposition
$v_h  = v_{h, 0} \cdot \chi_0 + v_{h, 1} \cdot \chi_1$ with $v_{h, i}
\in V_{h, i}^m (i=0, 1)$. 
From this decomposition, we
formally introduce a projection operator $(\cdot)^{\pi_i}: V_h^m
\rightarrow V_{h, i}^m$ such that $v_h^{\pi_i} := v_{h, i} \in  V_{h,
i}^m$ for any $v_h \in V_h^m$.
We define a seminorm $\Vhnorm{\cdot}$ on $V_h$ as
\begin{equation}
  \Vhnorm{v_h}^2 :=  \| \nabla v_h^{\pi_0} \|_{L^2(\Oho)}^2 + \|
  \nabla v_h^{\pi_1} \|_{L^2(\Ohl)}^2, \quad \forall v_h \in V_h^m,
  \label{eq_Vhnorm}
\end{equation}
which is clearly stronger than the seminorm $\|  \nabla v_h
\|_{L^2(\Omega_0 \cup \Omega_1)}$.
Moreover, we define a subspace $V_{h, \bc}^m := \{ v_h \in
V_h^m: \ v_h^{\pi_0}|_{\OhG} = v_h^{\pi_1}|_{\OhG}\}$. 
This subspace satisfies that $V_{h, \bc}^m \subset H_0^1(\Omega)$ and
the norm equivalence 
$\|\nabla v_h\|_{L^2(\Omega)} \leq \Vhnorm{v_h} \leq 2 \| \nabla v_h
\|_{L^2(\Omega)}$ holds for any $v_h \in V_{h, \bc}^m$.

Let $v$ and $\bm{q}$ be scalar- and vector-valued piecewise smooth
functions over $\Omega_0 \cup \Omega_1$.  The average
operator $\aver{\cdot}$ and the jump operator
$\jump{\cdot}$ on the interface $\Gamma$ are defined as 
\begin{displaymath}
  \aver{v}|_{\Gamma} := \frac{1}{2}(v^0|_{\Gamma} + v^1|_{\Gamma}),
  \quad \aver{\bm{q}}|_{\Gamma} := \frac{1}{2} (\bm{q}^0|_{\Gamma} +
  \bm{q}^1|_{\Gamma})
\end{displaymath}
with $v^0 := v|_{\Omega_0}, v^1 := v|_{\Omega_1}, 
\bm{q}^0 := \bm{q}|_{\Omega_0}, \bm{q}^1 := \bm{q}|_{\Omega_1}$, and
\begin{equation}
  \jump{v}|_{\Gamma} := (v^0|_{\Gamma} - v^1|_{\Gamma})\un_{\Gamma},
  \quad
  \jump{\bm{q}}|_{\Gamma} := (\bm{q}^0|_{\Gamma} - \bm{q}^1|_{\Gamma})
  \cdot \un_{\Gamma},
  \label{eq_jump}
\end{equation}
with $\un_{\Gamma}$ denoting the unit normal vector on $\Gamma$
pointing from $\Omega_0$ to $\Omega_1$. 

For a bounded domain $D$, we denote by $L^2(D)$ and $H^r(D)$ 
the standard Sobolev spaces, and their corresponding inner
products, seminorms and norms are also followed.  
From the embedding theory \cite{Girault1986finite}, we know that
$H^1(\Omega) \hookrightarrow H^{1/2}(\Gamma)$ equipped with the norm
$\|v \|_{H^{1/2}(\Gamma)} = \inf\limits_{w \in H^1(\Omega),
w|_{\Gamma} = v} \| w \|_{H^1(\Omega)}$. 
Let $H^{-1/2}(\Gamma)$ be the dual space of $H^{1/2}(\Gamma)$ with the
norm $\|v \|_{H^{-1/2}(\Gamma)} = \sup\limits_{0 \neq w \in
H^{1/2}(\Gamma)} \frac{(v, w)_{L^2(\Gamma)}}{\|w
\|_{H^{1/2}(\Gamma)}}$.
Hereafter, $C$ and $C$ with subscripts are denoted to be
generic positive constants that may vary in the context, but are
always independent of the mesh size $h$, and how the interface
$\Gamma$ cuts the mesh $\MTh$.

For any $K \in \MTh$, we define $\omega(K) := \{K' \in \MTh:
\overline{K'} \cap \overline{K} \neq \varnothing \}$ as the set of
elements touching $K$. 
We make the following assumptions to indicate the interface is
well-resolved by the background mesh \cite{Guzman2018infsup}:
\begin{assumption}
  The intersection $\overline{\OhG} \cap \partial \Omega$ is empty. 
  \label{as_empty}
\end{assumption}
\begin{assumption}
  For any $K \in \MThG$, we assume that both sets $W_K^0 := \omega(K)
  \cap \MThoc$ and $W_K^1 := \omega(K) \cap \MThlc$ are not empty. 
  \label{as_resolution}
\end{assumption}
The above assumptions can be fulfilled provided that the mesh $\MTh$ is
sufficiently fine.
Assumption \ref{as_resolution} allows us to introduce two mappings
$(\cdot)^{\varrho_i}: \MThG \rightarrow \MThic(i = 0, 1)$ for cut
elements. 
For any $K \in \MThG$, $K^{\varrho_i}$ can, in principle, be chosen as
anyone in $W^i_K$. In practice, one can take $K^{\varrho_i}$ as the
element in $W_K^i$ sharing a common face with $K$ if
possible.
Let $\bm{x}_K$ denote the barycenter of $K \in \MTh$. 
We denote by 
$B(\bm{z}, \xi)$ a disk (ball) centered at $\bm{z}$ with radius
$\xi$ and by $\partial B(\bm{z}, \xi)$ its boundary. 
Because $\MTh$ is quasi-uniform, there exists a constant $C_{\omega}$
independent of $h$ such that $\bigcup_{K' \in \omega(K)} \overline{K'}
\subset B(\bm{x}_K, C_{\omega} h_K)$ for any $K \in \MTh$.
We further define two balls $\Bw{K} := B(\bm{x}_K, C_{\omega} h_K)$
and $\Bt{K} := B(\bm{x}_K, h_K)$ for any $K \in \MTh$, which will be
used in the theoretical analysis. 
Since $K$ is fully contained in both balls, $\Bw{K}$ and $\Bt{K}$ 
will also be intersected by the interface if $K \in \MThG$ is a cut
element.

\begin{remark}
  Assumption \ref{as_resolution} can be further weakened by letting
  $W_K^i$ be formed by neighbours of $K$ with degree $L \geq 2$. 
  For example, $W_K^i$ can be set as $\omega(\omega(K)) \cap \MThic$,
  where $\omega(\omega(K)) := \bigcup\limits_{K' \in
  \omega(K)}\omega(K')$.
  \label{re_meshassump}
\end{remark}

We next introduce a local polynomial extension operator, 
which plays two roles in our analysis of the scheme. 
The first is for curing the effect caused by small cuts near the
interface. 
It is noted that the idea of the local extension has been a standard
stabilization strategy in unfitted finite element methods, see
\cite{Badia2018aggregated, Burman2021unfitted, Chen2021an,
Burman2022cutfem,
Huang2017unfitted, Johansson2013high, Yang2022an, Yang2024least} for
examples.
In our method, we integrate the polynomial extension with the ghost
penalty method \cite{Burman2010ghost} to construct suitable penalty
terms for stabilization, see Remark \ref{re_ghost}.
More importantly, the polynomial extension enables us to establish an
inverse-type estimate on the interface, which is crucial in deriving
the error estimates in the case of low regularity.

\begin{remark}
  The main idea in the ghost penalty method is to extend the control
  of the relevant norms from the interior domain to the entire domain
  of the active mesh by suitable bilinear forms \cite{Burman2010ghost,
  Massing2019stabilized}. 
  In our scheme, for $i = 0, 1$, we assume that the ghost penalty
  bilinear form $g_{h, i}(\cdot, \cdot)$ over $V_{h, i}^m \times V_{h,
  i}^m$ with the induced seminorm $\ghinorm{\cdot}^2 := g_{h,
  i}(\cdot, \cdot)$ satisfying the following properties exists. 
  \begin{enumerate}
    \item[\tu{P1}]: $H^1$-seminorm extension property: 
      \begin{equation}
        \| \nabla v_h \|_{L^2(\Ohi)} \leq C (\| \nabla v_h
        \|_{L^2(\Omega)} + \ghinorm{v_h}) \leq C \| \nabla v_h
        \|_{L^2(\Ohi)}, \quad \forall v_h \in
        V_{h, i}^m.
        \label{eq_ghP1}
      \end{equation}
    \item[\tu{P2}]: weak consistency: 
      \begin{equation}
        \ghinorm{\Pi_{V_{h, i}^m} v} \leq C h^t \| v
        \|_{H^{s+1}(\Omega)}, \quad \forall v \in H^{s+1}(\Omega),
        \quad t = \min(s, m),
        \label{eq_ghP2}
      \end{equation}
      where $\Pi_{V_{h, i}^m}: H^{s+1}(\Omega) \rightarrow V_{h, i}^m$
      is the Scott-Zhang interpolation operator
      \cite{Scott1990finite}.
  \end{enumerate}
  The properties $\tu{P1} - \tu{P2}$ are similar to Assumptions
  \tu{EP1} - \tu{EP2} given in \cite[\tu{EP1} -
  \tu{EP4}]{Massing2019stabilized}. 
  In \cite{Massing2019stabilized}, there are two extra assumptions
  (\tu{EP3} - \tu{EP4}) to ensure the condition number of the
  resulting linear system grows as in the standard finite element
  method, while in our method both assumptions are unnecessary.
  We also note that 
  the forms $g_{h, i}(\cdot, \cdot)$ that meet \tu{P1} - \tu{P2} can
  be constructed by the face-based penalties or the projection-based
  penalties, see \cite{Massing2019stabilized, Burman2015cutfem}. 
  \label{re_ghost}
\end{remark}

For any $K \in \MTh$, let $\Pi_K^m: L^2(K) \rightarrow \mb{P}_m(K)$
be the $L^2$ projection operator, and let $\mE_K^m:
\mb{P}_m(K) \rightarrow \mb{P}_m(\mb{R}^d)$ be the canonical extension
of a polynomial to $\mb{R}^d$.
For the analysis to the numerical scheme, we introduce two local
extension operators $E_{K, \omega}^m: L^2(K) \rightarrow
\mb{P}_m(\Bw{K})$ and $E_{K, \tau}^m: L^2(K) \rightarrow
\mb{P}_m(\Bt{K})$, which read 
\begin{equation}
  E_{K, \omega}^m v := (\mE_{K}^m (\Pi_K^m v))|_{B_{\omega(K)}},
  \quad E_{K, \tau}^m v := (\mE_{K}^m (\Pi_K^m v))|_{B_{\tau(K)}},
  \quad \forall K \in \MTh.
  \label{eq_EK}
\end{equation}
The difference between the two operators lies in the regions over
which the extension is performed. 
The first is used to construct the suitable ghost penalty terms, while
the second is employed to establish an inverse estimate on the
interface.
It is noted that both extension operators are mainly used for
theoretical analysis. 
For any $v \in \mb{P}_m(K)$, $E_{K, \omega}^m v = (\mE_K^m
v)|_{\Bw{K}}$ and $E_{K, \tau}^m v = (\mE_K^m v)|_{\Bt{K}}$ are just
direct extensions of $v$ to $\Bw{K}$ and $\Bt{K}$, respectively.
Hence, in the computer implementation both extension operators will
degenerate as canonical extensions. 
For any piecewise polynomial function $v_h \in V_{h, i}^m$, we have
that $E_{K, \omega}^m v_h = E_{K, \omega}^m (v_h|_K)$ and $E_{K,
\tau}^m v_h = E_{K, \tau}^m (v_h|_K)$.

We now show both extension operators are bounded in the sense of
satisfying the following estimates.
\begin{equation}
  \left.
  \begin{aligned}
    \| E_{K, \tau}^m v \|_{L^2(\Bt{K})} \\
    \|E_{K, \omega}^m v \|_{L^2(\Bw{K})}
  \end{aligned}
  \right\} 
  \leq C \| \Pi_K^m v \|_{L^2(K)}
  \leq C \| v \|_{L^2(K)}, \quad \forall v \in L^2(K), \quad \forall K
  \in \MTh. 
  \label{eq_EKL2}
\end{equation}
Because $K$ is shape-regular, there exists $B(\bm{x}_K,
\hat{\rho}_K) \subset K$ with radius $\hat{\rho}_K \geq \hat{C} h_K$.
The norm equivalence on the finite dimensional space gives us that $\|
v \|_{L^2(B(\bm{0}, C_{\omega} \hat{C}))} \leq C \| v
\|_{L^2(B(\bm{0},1))}$ for any $v \in \mb{P}_m(B(\bm{0},
C_{\omega} \hat{C}))$. 
Combining with the affine mapping from $B(\bm{0}, 1)$ to $B(\bm{x}_K,
\hat{\rho}_K)$ and the property of $\Pi_K^m$, we find that
\begin{align*}
  \| E_{K, \omega}^m v \|_{L^2(\Bw{K})} \leq C \| E_{K, \omega}^m v
  \|_{L^2(B(\bm{x}_K, \hat{\rho}_K))} \leq C \| \Pi_K^m v\|_{L^2(K)}
  \leq C \| v \|_{L^2(K)}, \quad \forall v \in L^2(K),
\end{align*}
which yields the first estimate in \eqref{eq_EKL2}, and the bound of
$\| E_{K, \tau}^m v\|$ can be derived analogously.
Moreover, a similar argument leads to the estimate
\begin{equation}
  \left.
  \begin{aligned}
    \|\nabla  E_{K, \tau}^m v \|_{L^2(\Bt{K})} \\
    \|\nabla  E_{K, \omega}^m v \|_{L^2(\Bw{K})}
  \end{aligned} \right\} \leq C \| \nabla v \|_{L^2(K)}, \quad \forall
  v \in \mb{P}_m(K).
  \label{eq_EKH1}
\end{equation}
Based on $E_{K, \omega}^m$, 
we outline a method to construct proper penalty forms $g_{h,
i}(\cdot, \cdot)$ that satisfy \tu{P1} - \tu{P2} in Remark
\ref{re_ghost}. 
For $i = 0, 1$, we define
\begin{equation}
  g_{h, i}(v_h, w_h) := \sum_{K \in \MThG} \int_K \nabla (v_h -
  E_{K^{\varrho_i}, \omega}^m v_h) \cdot \nabla (w_h - E_{
  K^{\varrho_i}, \omega}^m w_h) \dbx{x}, \quad \forall v_h, w_h \in
  V_{h, i}^m,
  \label{eq_ghi}
\end{equation}
with the seminorm $\ghinorm{v_h}^2 := g_{h, i}(v_h, v_h)$.
Since $K \in B_{\omega(K^{\varrho_i})}$ for any $K \in \MThG$, the
polynomial $E_{K^{\varrho_i}, \omega}^m v_h$ is well-defined on $K$.
From \eqref{eq_EKH1} and the triangle inequality, we obtain that
\begin{align*}
  \ghinorm{v_h}^2 & \leq C \sum_{K \in \MThG} ( \| \nabla v_h
  \|_{L^2(K)}^2 + \| \nabla E_{K^{\varrho_i}, \omega}^m v_h
  \|_{L^2(K)}^2) \\
  & \leq C  \sum_{K \in \MThG} ( \| \nabla v_h
  \|_{L^2(K)}^2 + \| \nabla E_{K^{\varrho_i}, \omega}^m v_h
  \|_{L^2(\Bw{K^{\varrho_i}})}^2) \\
  & \leq C \sum_{K \in \MThG} ( \| \nabla v_h \|_{L^2(K)}^2 + \|
  \nabla v_h \|_{L^2(K^{\varrho_i})}^2) \leq C \| \nabla v_h
  \|_{L^2(\Ohi)}^2, \quad \forall v_h \in V_{h, i}^m,
\end{align*}
which indicates the upper bound of \eqref{eq_ghP1}. 
We again apply the triangle inequality to deduce that
\begin{align*}
  & \| \nabla v_h \|_{L^2(\MThG)}^2  = \sum_{K \in \MThG} \| \nabla v_h
  \|_{L^2(K)}^2 \\
  & \leq C \sum_{K \in \MThG} (\|\nabla (v_h - E_{
  K^{\varrho_i}, \omega}^m v_h) \|_{L^2(K)}^2 + \| \nabla E_{
  K^{\varrho_i}, \omega}^m
  v_h\|_{L^2(K)}^2) \\
  & \leq C \big( g_{h, i}(v_h, v_h) + \sum_{K \in \MThG}  \|\nabla
  E_{K^{\varrho_i}, \omega}^m v_h\|_{L^2(\Bw{K^{\varrho_i}})}^2
  \big) \\
  & \leq C \big( g_{h, i}(v_h, v_h) + \sum_{K \in \MThG} \| \nabla
  v_h\|_{L^2(K^{\varrho_i})}^2 \big) \leq C \big(g_{h, i}(v_h, v_h) +
  \| \nabla v_h \|_{L^2(\Ohic)}^2 \big), \quad \forall v_h \in V_{h,
  i}^m,
\end{align*}
which gives the lower bound of \eqref{eq_ghP1}.
We next turn to the property \tu{P2}.
Given any $v \in H^{s+1}(\Omega)$, for any $K \in \MThG$, there exists
$v_{K, \omega} \in \mb{P}_m(\Bw{K^{\varrho_i}})$ such that 
\begin{displaymath}
  \| \nabla( v - v_{K, \omega}) \|_{L^2(\Bw{K^{\varrho_i}})} \leq C
  h_K^{t} \| v\|_{H^{s+1}(\Bw{K^{\varrho_i}})}, \quad t = \min(s, m).
\end{displaymath}
Let $v_h :=  \Pi_{V_{h, i}^m} v$ and set $v_K := v_h|_K$, there holds
\begin{align*}
  &\ghinorm{v_h}^2  = \sum_{K \in \MThG} \| \nabla(v_h  - E_{
  K^{\varrho_i}, \omega}^m v_h) \|_{L^2(K)}^2 = \sum_{K \in \MThG} \|
  \nabla( v_K - E_{K^{\varrho_i}, \omega}^m v_{K^{\varrho_i}})
  \|_{L^2(K)}^2 \\
  & \leq C \sum_{K \in \MThG} (\| \nabla (v_K - v_{K, \omega})
  \|_{L^2(K)}^2 +  \| \nabla (E_{ K^{\varrho_i}, \omega}^m (v_{K,
  \omega} - v_{K^{\varrho_i}})) \|_{L^2(K)}^2) \\
  & \leq C \sum_{K \in \MThG} ( \| \nabla (v_K - v_{K, \omega})
  \|_{L^2(K)}^2 +  \| \nabla(v_{K, \omega} -  v_{K^{\varrho_i}})
  \|_{L^2(K^{\varrho_i})}^2) \\ 
  & \leq C  \sum_{K \in \MThG} ( \| \nabla v - \nabla v_{K, \omega}
  \|_{L^2(\Bw{K})}^2 + \|\nabla v - \nabla v_K \|_{L^2(K)}^2 + \|
  \nabla v - \nabla v_{K^{\varrho_i}} \|_{L^2(K^{\varrho_i})}^2) \\
  & \leq C h^{2t} \| v \|_{H^{s+1}(\Omega)}^2.
\end{align*}
The estimate \eqref{eq_ghP2} is reached.
In the numerical scheme, the bilinear form $g_{h, i}(\cdot, \cdot)$ is
employed to guarantee both the stability near the interface and the
uniform upper bound of the condition number to the final linear
system.

We next establish an inverse-type estimate on the interface for every
cut $K \in \MThG$, based on the disk $\Bt{K}$.
We begin with the case where $\Gamma$ is $C^2$. 
For any cut $K \in \MThG$, $\Bt{K}$ is also cut by $\Gamma$, and we
let $\Gamma_{\Bt{K}} := \Gamma \cap \Bt{K}$ here.

\begin{lemma}
  For the $C^2$ interface $\Gamma$, there exists a constant $h_0$
  such that for any $h < h_0$, there holds 
  \begin{equation}
    \| w \|_{L^{\infty}(\Gamma_K)} \leq C h_K^{(1-d)/2} (  \| w
    \|_{L^2(\Gamma_{\Bt{K}})} + h_K^{1/2} \| \nabla w
    \|_{L^2(K)}), \quad \forall w \in \mb{P}_m(\Bt{K}), \quad
    \forall K \in \MThG.
    \label{eq_C2inverse}
  \end{equation}
  \label{le_C2inverse}
\end{lemma}
\begin{proof}
  We first prove \eqref{eq_C2inverse} for the two-dimensional case. 
  For sufficiently small $h$,  $\Gamma$ intersects $\partial \Bt{K}$
  at two distinct points, see $\bm{x}_0,
  \bm{x}_1$ in Fig.~\ref{fig_interfacecircle2d}.
  Let $e_0$ be the edge connecting $\bm{x}_0$ to $\bm{x}_1$, 
  and we let $\un_0$ and $\un$ be the unit outward normal vectors on
  $e_0$ and $\Gamma_{\Bt{K}}$, respectively. 
  For sufficiently small $h$, there holds $\frac{1}{2} \leq |\un_0
  \cdot \un(\bm{x})| \leq 1$ for any $\bm{x} \in 
  \Gamma_{\Bt{K}}$ \cite{Huang2017unfitted, Chen1998interface}.
  Denoting by $B^0_{\tau}$ the domain enclosed by $e_0$ and the curve
  $\Gamma_{\Bt{K}}$, we have that $|B^0_{\tau}| \leq C h_K^3$
  and $\text{width}(B^0_{\tau}) \leq Ch_K^2$.
  Let $\bm{y}$ be any point on $e_0$, there holds $|\bm{x}_K - \bm{y}|
  < \frac{2h_K}{3}$ (see the property \eqref{eq_app_bc} in Appendix), 
  which implies the distance between $\bm{x}_K$ and $e_0$ is less than
  $\frac{2h_K}{3}$.
  Since $\Bt{K}$ has radius $h_K$, we know that $|e_0| \geq C
  h_K$ for some constant $C > 0$. 
  Combining with the inverse estimate, we find that $\|w
  \|_{L^{\infty}(e_0)} \leq C h_K^{-1/2} \| w \|_{L^2(e_0)}$. 
  We let $\bm{z} \in \Gamma_{\Bt{K}}$ be the point such that
  $|w(\bm{z})| = \| w\|_{L^{\infty}(\Gamma_{\Bt{K}})}$, and then 
  there exists $\bm{\xi}_{\bm{z}} \in e_0$ such that $|\bm{z} -
  \bm{\xi}_{\bm{z}}| \leq Ch_K^2$.
  We derive that 
  \begin{align}
    \| w \|_{L^{\infty}(\Gamma_{\Bt{K}})} & \leq |
    w(\bm{\xi}_{\bm{z}})| + Ch_K^2 \| \nabla w\|_{L^{\infty}(\Bt{K})} \leq
    \| w \|_{L^{\infty}(e_0)} + Ch_K \| \nabla w\|_{L^2(\Bt{K})} \nonumber
    \\
    & \leq \| w \|_{L^{\infty}(e_0)} + Ch_K \| \nabla w\|_{L^2(K)}
    \leq C h_K^{-1/2} ( \| w \|_{L^2(e_0)} + h_K^{3/2} \| \nabla w
    \|_{L^2(K)}). \label{eq_wLinfty}
  \end{align}
  Let $\bm{v} := w \un_0$ be a vector-valued function. 
  We obtain that
  \begin{align*}
    \int_{e_0} \un_0 \cdot \bm{v} w \dx{s} + 
    \int_{\Gamma_{\Bt{K}}} \un \cdot \bm{v} w \dx{s} & =
    \int_{B_{\tau}^0} \bm{v} \cdot \nabla w \dx{x} + 
    \int_{B_{\tau}^0} \nabla \cdot \bm{v} w \dx{x} \\
    & \leq Ch_K^3 \| w \|_{L^{\infty}(B_{\tau}^0)} \| \nabla  w
    \|_{L^{\infty}(B_{\tau}^0)},
  \end{align*}
  which brings that 
  \begin{align}
    \| w\|_{L^2(e_0)}^2 & \leq C ( \| w \|_{L^2(\Gamma_{\Bt{K}})}^2 + 
    h_K^{3} \| w \|_{L^{\infty}(B_{\tau}^0)} \| \nabla  w
    \|_{L^{\infty}(B_{\tau}^0)}) 
    \label{eq_wL2} \\
    & \leq C ( \| w \|_{L^2(\Gamma_{\Bt{K}})}^2 + 
    h_K^{3} \| w \|_{L^{\infty}(B_{\tau}^0)}^2 + h_K^3 \| \nabla  w
    \|_{L^{\infty}(B_{\tau}^0)}^2). \nonumber \\
    & \leq C ( \| w \|_{L^2(\Gamma_{\Bt{K}})}^2 + 
    h_K^{3} \| w \|_{L^{\infty}(B_{\tau}^0)}^2 + h_K \| \nabla  w
    \|_{L^{2}(K)}^2). \nonumber
  \end{align}
  Similar to \eqref{eq_wLinfty}, we have that 
  \begin{align*}
    \| w \|_{L^{\infty}(B_{\tau}^0)} \leq \| w
    \|_{L^{\infty}(\Gamma_{\Bt{K}})} + C h_K^2 \| \nabla w
    \|_{L^{\infty}(\Bt{K})} \leq  \| w
    \|_{L^{\infty}(\Gamma_{\Bt{K}})}  + C h_K \| \nabla w \|_{L^2(K)}.
  \end{align*}
  Collecting all above estimate gives that
  \begin{equation}
    \|w \|_{L^{\infty}(\Gamma_{\Bt{K}})} \leq C h_K^{-1/2} ( \| w
    \|_{L^{2}(\Gamma_{\Bt{K}})} + h_K^{1/2}  \| \nabla w \|_{L^{2}(K)}
    ) + Ch_K \| w \|_{L^{\infty}(\Gamma_{\Bt{K}})}.
    \label{eq_wLinftyGamma}
  \end{equation}
  For sufficiently small $h$, the desired estimate
  \eqref{eq_C2inverse} is reached in two dimensions.

  \begin{figure}[htp]
    \centering
    \begin{minipage}[t]{0.31\textwidth}
      \centering
      \begin{tikzpicture}[scale=1.25]
        \coordinate (x0) at (-0.95, 0.639); 
        \coordinate (x1) at (0.98, 0.175); 
        \coordinate (x2) at (1.4625, 0.059); 
        \coordinate (y0) at (-0.718, 1.60);
        \draw[thick] (0, 0.35) circle [radius=1]; 
        \draw[thick, red] (x0) [in = 130, out = 25] to
        (x1);
        \draw[thick, red] (-1.25, 0.5) [in = 205, out = 30] to
        (x0);
        \draw[thick, red] (x1) [out = -50, in = 100]
        to (1.2, -0.2);
        \draw[thick] (-0.6, 0.3) -- (0.5, 0) -- (0.2, 1.1) -- 
        (-0.6, 0.3); 
        \node at (0, 0.5) {\small $K$};
        \node at (0.2, -0.3) {\small $\Bt{K}$};
      \end{tikzpicture}
    \end{minipage}
    \hfill
    \begin{minipage}[t]{0.31\textwidth}
      \centering
      \begin{tikzpicture}[scale=1.25]
        \coordinate (x0) at (-0.95, 0.639); 
        \coordinate (x1) at (0.98, 0.175); 
        \coordinate (x2) at (1.4625, 0.059); 
        \coordinate (y0) at (-0.718, 1.60);
        \draw[thick] (0, 0.35) circle [radius=1]; 
        \draw[thick, red] (x0) [in = 130, out = 25] to
        (x1);
        \draw[thick, red, dashed] (-1.25, 0.5) [in = 205, out = 30] to
        (x0);
        \draw[thick, red, dashed] (x1) [out = -50, in = 100]
        to (1.2, -0.2);
        \draw[thick, fill=black] (x0) circle [radius=0.035];
        \draw[thick, fill=black] (x1) circle [radius=0.035];
        \draw[thick, ->] (x0) -- (x2);
        \draw[thick, ->] (x0) -- (y0);
        \node[below] at ($(x0) + (-0.25, 0.3)$) {\small $\bm{x}_0$};
        \node[] at ($(x1) + (0.25, 0.1)$) {\small $\bm{x}_1$};
        \node[below] at ($0.5*(x0) + 0.5*(x1)$) {\small $e_0$};
        \node at (0.2, -0.3) {\small $\Bt{K}$};
        \node[below] at ($(0.7, 1.) + (-0.8, 0.2)$) {\small
        $\Gamma_{\Bt{K}}$};
      \end{tikzpicture}
    \end{minipage}
    \hfill
    \begin{minipage}[t]{0.31\textwidth}
      \centering
      \begin{tikzpicture}[scale=1.25]
        \draw[thick, dashed] (0, 0) -- (3, 0); 
        \draw[thick, red] (0, 0) [out = 50, in = 130] to (2.5, 0);
        \draw[thick] (0, 0) -- (2.5, 0);
        \draw[thick, dashed] (0, 0) -- (0, 1.5); 
        \draw[thick, dashed] (0, 0) -- (3.2, 0); 
        \node[below] at (0, 0) {\footnotesize $\bm{x}_0$};
        \node[below] at (2.6, 0) {\footnotesize $\bm{x}_1$};
        \node[above] at(1., 0.5) {\footnotesize $\Gamma_{\Bt{K}}$};
        \draw[thick, fill=black] (0, 0) circle [radius=0.03];
        \draw[thick, fill=black] (2.5, 0) circle [radius=0.03];
        \draw[thick, <->] (0.85, 0) -- (0.85, 0.50);
        \node[right] at (0.85, 0.235) {\footnotesize $O(h_K^2)$};
        \draw[thick, ->] (1.95, 0) -- (1.95, -0.5);
        \node[right] at (1.95, -0.25) {$\un_0$};
        \draw[thick, ->] (2.15, 0.3) -- (2.43, 0.65);
        \node[left] at ($(2.325, 0.475) + (0.05, 0.1)$) {$\un$};
      \end{tikzpicture}
    \end{minipage}
    \caption{The interface intersects $\Bt{K}$ in two dimensions.}
    \label{fig_interfacecircle2d}
  \end{figure}
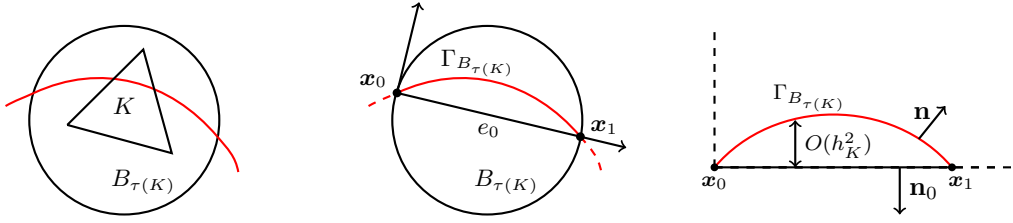

  The proof in three dimensions is similar. 
  For sufficiently small $h$, the intersection $\Gamma_{\Bt{K}}$ is
  contained in a strip of width $\delta \leq C h_K^2$ \cite[Lemma
  3.4]{Huang2017unfitted}. 
  We denote by $L_0$ and $L_1$ two parallel planes that contain the
  strip, see Fig.~\ref{fig_interfacecircle3d}, and the distance
  between two planes is less than $C h_K^2$. 
  Let $L_0$ be the one that is closer to $\bm{x}_K$, 
  and we let $B_0 := \Bt{K} \cap L_0$ and 
  $B_1 := \Bt{K} \cap L_1$. 
  We further define $B_{\tau}^0$ as the domain enclosed by $B_0$,
  $B_1$ and the ball $\Bt{K}$.
  We denote by $\bm{y}$ any point on $\Gamma_K$.  From
  \eqref{eq_app_bc}, it follows that $|\bm{y} - \bm{x}_K| \leq \frac{3
  h_K}{4}$. 
  For sufficiently small $h$, each $B_j$ is a disk with radius
  $r_j$ satisfying $r_j \geq C h_K$ for some constant $C > 0$.
  Let $\hat{B}_1$ and $U_0$ be the projection of $B_1$ and
  $\Gamma_{\Bt{K}}$ on $B_0$, respectively, and there holds
  $\hat{B}_1 \subset U_0 \subset B_0$.
  We have the following inverse estimate: 
  \begin{equation}
    \|v \|_{L^{\infty}(U_0)} \leq \| v \|_{L^{\infty}(B_0)} \leq C
    h_K^{-1} \| v \|_{L^2(B_0)} \leq C h_K^{-1} \| v
    \|_{L^2(\hat{B}_1)}, \quad \forall v \in \mb{P}_m(\Bt{K}).
    \label{eq_B01inverse}
  \end{equation}
  We let $\un_0, \un$ be the unit outward normal vectors on $B_0$ and
  $\Gamma_{\Bt{K}}$, respectively. 
  For small enough $h$, there holds $\frac{1}{2} \leq |\un_0 \cdot
  \un(\bm{x})| \leq 1$ for any $\bm{x} \in \Gamma_{\Bt{K}}$.
  Similar to \eqref{eq_wL2} and \eqref{eq_wLinfty}, we observe that
  \begin{align*}
    \|w \|_{L^{\infty}(\Gamma_{\Bt{K}})} \leq \| w
    \|_{L^{\infty}(U_0)} + C h_K^2 \| \nabla w \|_{L^{\infty}(\Bt{K})} \leq
    C h_K^{-1} ( \| w \|_{L^2(\hat{B}_1)} + h_K^{3/2} \| \nabla w
    \|_{L^2(K)}),
  \end{align*}
  and
  \begin{align*}
    \| w \|_{L^2(\hat{B}_1)}^2 \leq C( \| w
    \|_{L^2(\Gamma_{\Bt{K}})}^2 + h_K^3 \| w
    \|_{L^{\infty}(B_{\tau}^0)}^2 +  h_K^4 \| w
    \|_{L^{\infty}(B_{\tau}^0)} \| \nabla w
    \|_{L^{\infty}(B_{\tau}^0)}).
  \end{align*}
  As in the two-dimensional case, we have that
  \begin{displaymath}
    \| w
    \|_{L^{\infty}(B_{\tau}^0)} \leq \| w
    \|_{L^{\infty}(\Gamma_{\Bt{K}})} + C h_K^2 \| \nabla w
    \|_{L^{\infty}(B_{\tau}^0)} \leq  \| w
    \|_{L^{\infty}(\Gamma_{\Bt{K}})} + C h_K^{1/2} \| \nabla w
    \|_{L^{2}(K)}.
  \end{displaymath}
  Combining all above estimates yields that 
  \begin{displaymath}
    \|w \|_{L^{\infty}(\Gamma_{\Bt{K}})} \leq C h_K^{-1} ( \| w
    \|_{L^{2}(\Gamma_{\Bt{K}})} + h_K^{1/2}  \| \nabla w \|_{L^{2}(K)}
    ) + Ch_K \| w \|_{L^{\infty}(\Gamma_{\Bt{K}})},
  \end{displaymath}
  which leads to the estimate \eqref{eq_C2inverse} in three
  dimensions with a sufficiently small $h$.
  The proof is complete.

  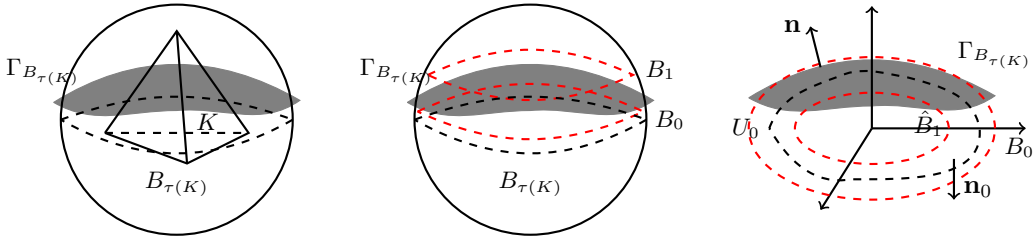
\begin{figure}[htp]
    \begin{minipage}[t]{0.31\textwidth}
      \centering
      \begin{tikzpicture}[scale=1.35]
        \coordinate (A0) at (-0.7, -0.0);
        \coordinate (A1) at (0.1, -0.3);
        \coordinate (A2) at (0.7, -0.0);
        \coordinate (A3) at (0., 1);
        \draw[fill, gray] (-1.2, 0.3) to [out=32, in = 150] (1.2,
        0.32) to [out = 215, in = 5] (0, 0.22) 
        to [out = 180, in = -30] (-1.2, 0.3);
        \draw[thick] (0, 0.13) circle [radius=1.13]; 
        \draw[thick] (A0) -- (A1) -- (A2) -- (A3) -- (A0);
        \draw[thick] (A1) -- (A3);
        \draw[thick, dashed] (A0) -- (A2);
        \node[above] at ($0.5*(A1) + 0.5*(A2) + (-0.1, 0.09)$) {\small $K$};
        \node at (0, -0.5) {\small $\Bt{K}$};
        \node at (-1.3, 0.6) {\small $\Gamma_{\Bt{K}}$};
        \draw[thick, dashed] (-1.13, 0.13) to [out = 20, in = 160]
        (1.13, 0.13);
        \draw[thick, dashed] (1.13, 0.13) to [out = 210, in = -30]
        (-1.13, 0.13);
      \end{tikzpicture}
    \end{minipage}
    \hfill
    \begin{minipage}[t]{0.31\textwidth}
      \centering
      \begin{tikzpicture}[scale=1.35]
        \draw[fill, gray] (-1.2, 0.3) to [out=32, in = 150] (1.2,
        0.32) to [out = 215, in = 5] (0, 0.22) 
        to [out = 180, in = -30] (-1.2, 0.3);
        \draw[thick] (0, 0.13) circle [radius=1.13]; 
        \draw[thick, dashed, red] (-1., 0.57) [out=25, in = 155] to
        (1, 0.57) [out = 205, in = -25] to (-1., 0.57);
        \draw[thick, dashed, red] (-1.1, 0.21) [out=25, in = 155] to
        (1.1, 0.21) [out = 205, in = -25] to (-1.1, 0.21);
        \node at (0, -0.5) {\small $\Bt{K}$};
        \node at (-1.3, 0.6) {\small $\Gamma_{\Bt{K}}$};
        \node[right] at (1.12, 0.13) {\small $B_0$};
        \node[right] at (1.03, 0.62) {\small $B_1$};
        \draw[thick, dashed] (-1.13, 0.13) to [out = 20, in = 160]
        (1.13, 0.13);
        \draw[thick, dashed] (1.13, 0.13) to [out = 210, in = -30]
        (-1.13, 0.13);
      \end{tikzpicture}
    \end{minipage}
    \hfill
    \begin{minipage}[t]{0.31\textwidth}
      \centering
      \begin{tikzpicture}[scale=1.35]
        \draw[white] (0, -0.85) circle [radius=0.2];
        \draw[fill, gray] (-1.2, 0.3) to [out=32, in = 150] (1.2,
        0.32) to [out = 215, in = 5] (0, 0.22) 
        to [out = 180, in = -30] (-1.2, 0.3);
        \draw[thick, ->] (0, 0) -- (0, 1.2); 
        \draw[thick, ->] (0, 0) -- (1.5, 0); 
        \draw[thick, ->] (0, 0) -- (-0.5, -0.8);
        \draw[thick, dashed, red] ellipse (0.75 and 0.35);
        \draw[thick, dashed, red] ellipse (1.2 and 0.7);
        \node[below right] at (1.2, 0) {\small $B_0$};
        \node[] at (0.55, 0.05) {\small $\hat{B}_1$};
        \draw[thick, dashed] (-1, 0) to [out = 70, in = 200] (-0.35,
        0.5) to [out = 25, in = 170] (0.35, 0.5) to [out = -15, in =
        100] (1.05, 0) to [out = -85, in = 0] (-0.2, -0.5) to
        [out=190, in = -60] (-1, 0);
        \node[left] at (-1, 0) {\small $U_0$};
        \node[above] at (1.2, 0.5) {\small $\Gamma_{\Bt{K}}$};
        \draw[thick, ->] (0.8, -0.3) -- (0.8, -0.7);
        \node[right] at (0.8, -0.6) {$\un_0$};
        \draw[thick, ->] (-0.5, 0.6) -- (-0.6, 1);
        \node[left] at (-0.6, 1) {$\un$};
      \end{tikzpicture}
    \end{minipage}
    \caption{The interface intersects $\Bt{K}$ in three dimensions.}
    \label{fig_interfacecircle3d}
  \end{figure}
\end{proof}
For the case that $\Gamma$ is polygonal (polyhedral), let  
$\Gamma_j(1 \leq j \leq J)$ denote the sides of $\Gamma$, 
where each $\Gamma_j$ is a line segment (polygon). 
For every $K \in \MThG$, we let $\Gamma_{K, k_j}(1 \leq j \leq J_K)$ 
be the sides intersecting $K$, where $\Gamma_{K, k_j}$ are parts of
$\Gamma_{k_j}$. 
In this case, we define $\Gamma_{\Bt{K}, k_j} := \Gamma_{k_j} \cap
\Bt{K}$ and define $\Gamma_{\Bt{K}} := \cup_{1 \leq j \leq J_K}
\Gamma_{\Bt{K}, k_j}$.
\begin{lemma}
  If $\Gamma$ is polygonal (polyhedral), there exists a constant $h_1$
  such that for any $h < h_1$, there holds 
  \begin{equation}
    \| w \|_{L^{\infty}(\Gamma_K)} \leq C  h_K^{(1 - d)/2}
    \| w \|_{L^2(\Gamma_{\Bt{K}})}, \quad \forall w \in
    \mb{P}_m(B_{\tau(K)}), \quad \forall K \in \MThG.
    \label{eq_polyinverse}
  \end{equation}
  \label{le_polyinverse}
\end{lemma}
\begin{proof}
  In two dimensions, for sufficiently small $h$, every side
  $\Gamma_{k_j}$ will intersect the circle $\partial \Bt{K}$, and we
  denote by $\bm{z}$ any point in the intersection. 
  Let $\bm{y}$ be any point on $\Gamma_{K, k_j}$, we have that 
  $|\bm{y} - \bm{x}_K| \leq \frac{2 h_K}{3}$ by \eqref{eq_app_bc}.
  From $|\bm{x}_K - \bm{z}| = h_K$,  we find that $|\bm{y} - \bm{z}|
  \geq \frac{h_K}{3}$, which implies $|\Gamma_{\Bt{K}, k_j}| =
  | \Gamma_{k_j} \cap \Bt{K}| \geq |\bm{y} - \bm{z}| \geq C h_K$. 
  From the inverse estimate, we have that 
  $\| w \|_{L^{\infty}(\Gamma_{K, k_j})} \leq C h_K^{-1/2} \| w
  \|_{L^2(\Gamma_{\Bt{K}, k_j})}$. 
  Then,
  the estimate \eqref{eq_polyinverse} is reached in two dimensions.
  In three dimensions, it is similar to conclude that $|
  \Gamma_{\Bt{K}, k_j}| \geq C h_K^2$, and then combining this with
  the inverse estimate on $\Gamma_{\Bt{K}, k_j}$ yields the desired
  estimate. 
  The proof is complete.
\end{proof}
Consequently, from \eqref{eq_C2inverse} and
\eqref{eq_polyinverse}, we have that for any $K \in \MThG$ and any
$v_h \in V_h^m$, there holds
\begin{equation}
  \| \jump{v_h} \|_{L^{\infty}(\Gamma_K)} \leq Ch_K^{(1-d)/2} (\|
  \jump{E_{K, \tau}^m v_h} \|_{L^2(\Gamma_{\Bt{K}})} + h_K^{1/2} (\|
  \nabla v_h^{\pi_0} \|_{L^2(K)} +  \|
  \nabla v_h^{\pi_1} \|_{L^2(K)})),
  \label{eq_jumpvh}
\end{equation}
where the jump $\jump{E_{K, \tau}^m v_h}|_{\Gamma_{\Bt{K}}} :=
((E_{K, \tau}^m v_h^{\pi_0})|_{\Gamma_{\Bt{K}}} -  (E_{K, \tau}^m
v_h^{\pi_1})|_{\Gamma_{\Bt{K}}})\un_{\Gamma}$.  
We define a bilinear form $s_h(\cdot, \cdot)$ on $V_h^m$ as
\begin{equation}
  s_h(v_h, w_h) := \sum_{K \in \MThG} \int_{\Gamma_{\Bt{K}}} h_K^{-1}
  \jump{E_{K, \tau}^m v_h} \cdot \jump{E_{K, \tau}^m w_h} \dbx{s},
  \quad \forall v_h, w_h \in V_h^m,
  \label{eq_shform}
\end{equation}
with the induced seminorm $\shnorm{v_h}^2 := s_h(v_h, v_h)$ for any
$v_h \in V_h^m$.

\begin{remark}
  The estimates \eqref{eq_jumpvh} and \eqref{eq_Vhdiff} are
  fundamental to our numerical analysis, especially for the low
  regularity case that $s < 1$.
  Both estimates strongly rely on the condition that the intersection 
  $\Gamma_{\Bt{K}}$ ($\Gamma_{\Bt{K}, k_j}$) satisfies 
  $|\Gamma_{\Bt{K}}| \geq C h_K^{d-1}$.  
  Roughly speaking, the set $\Gamma_{\Bt{K}}$ can be replaced by any
  set $\widetilde{\Gamma}_{K}$ satisfying $\Gamma_K \subset
  \widetilde{\Gamma}_K$ and $C_0 h_K^{d-1} \leq |
  \widetilde{\Gamma}_K| \leq C_1 h_K^{d-1}$.
  In addition to exact integration over $\Gamma_{\Bt{K}}$, 
  a suitable set $\widetilde{\Gamma}_K$ can also be constructed using
  the neighbouring elements of $K$.
  One can first assign a parameter $\kappa \in (0, 1)$, and for the
  cut element $K$, let $\mS(K) = \{K\}$ first. 
  If $|\Gamma_{\mS(K)}|(\Gamma_{\mS(K)} := \cup_{K' \in \mS(K)}
  \Gamma_{K'}) < \kappa h_K$, then seek a cut element $K'$ that is
  adjacent to an element in $\mS(K)$ and add it to $\mS(K)$.
  Repeat this process until $|\Gamma_{\mS(K)}| \geq \kappa h_K$, and
  finally set $ \widetilde{\Gamma}_K = \Gamma_{\mS(K)}$.
  \label{re_GammaK}
\end{remark}

From the estimates \eqref{eq_Vhnorm} and \eqref{eq_jumpvh}, we state
the following result. 
\begin{lemma}
  For any $w_h \in V_h^m$, there exists  $w_{h,\bc} \in V_{h, \bc}^m$
  such that for $i = 0, 1$, there holds
  \begin{equation}
    \|w_h^{\pi_i} - w_{h, \bc}^{\pi_i} \|_{L^2(\Ohi)} + h \|
    \nabla( w_h^{\pi_i} - w_{h, \bc}^{\pi_i}) \|_{L^2(\Ohi)} \leq
    Ch (\shnorm{w_h} + \Vhnorm{w_h}).
    \label{eq_Vhdiff}
  \end{equation}
  \label{le_Vhdiff}
\end{lemma}
\begin{proof}
  For the submesh $\MThi(i = 0, 1)$, we let $\mN_{h, i} := \{
  \bu_j^i\}_{j=0}^{n_{m, i}}$ be the set of Lagrange nodes associated
  with $\MThi$, and let $\{ \varphi_{\bu_j^i}\}_{j=0}^{n_{m, i}}$
  denote Lagrange basis functions corresponding to the node $\bu_j^i$,
  i.e. $\varphi_{\bu_j^i}(\bu_k^i) = \delta_{jk}$.
  Then, $V_{h, i}^m = \text{span}(\{\varphi_{\bu_j^i}\}_{j=0}^{n_{m,
  i}})$.
  Since the finite element function $w_h$ has the decomposition
  $w_h = w_h^{\pi_0} \cdot \chi_0 + w_h^{\pi_1} \cdot \chi_1$,
  we will construct a new function $\hat{w}_{h, i} \in V_{h, i}^m$
  from $w_h^{\pi_i}$ using the following nodal values, 
  \begin{equation}
    \begin{aligned}
      \hat{w}_{h, i}(\bu) &= w_{h}^{\pi_i}(\bu), && \forall \bu \in
      \overline{\Omega}_{h, i} \backslash \overline{\OhG}, \\
      \hat{w}_{h, i}(\bu) &= \frac{1}{2}(w_h^{\pi_0}(\bu) +
      w_h^{\pi_1}(\bu)), && \forall \bu \in
      \overline{\OhG}, 
    \end{aligned}
    \label{eq_newhatw}
  \end{equation}
  We define $w_{h, \bc}:= \hat{w}_{h, 0} \cdot \chi_0 + \hat{w}_{h, 1}
  \cdot \chi_1$. 
  It is straightforward to see that $\hat{w}_{h, 0}|_{\OhG} =
  \hat{w}_{h,1}|_{\OhG}$. 
  By Assumption \ref{as_empty}, we have $w_{h, \bc}|_{\partial \Omega}
  = 0$.  Thus, it follows that $w_{h, \bc} \in V_{h, \bc}^m$.
  It remains to estimate the bound of $w_h - w_{h, \bc}$ as
  \eqref{eq_Vhdiff}. 
  From \eqref{eq_newhatw}, we have that 
  \begin{align*}
    \| w_{h, \bc}^{\pi_i} - w_h^{\pi_i} \|_{L^2(\Ohi)}^2 & \leq
    C \sum_{\bu \in \overline{\OhG}} |w_h^{\pi_i}(\bu) -
    \hat{w}_{h, i}(\bu)|^2 \| \varphi_{\bu} \|_{L^2(\Ohi)}^2 \\
    & \leq C  \sum_{\bu \in \overline{\OhG}} |w_h^{\pi_0}(\bu) -
    w_h^{\pi_1}(\bu)|^2 \| \varphi_{\bu} \|_{L^2(\Ohi)}^2 
  \end{align*}
  For any node $\bu \in \overline{\OhG}$, we let $K_{\bu} \in \MThG$
  be the element with $\bu \in \overline{K}_{\bu}$. 
  We know that $\| \varphi_{\bu} \|_{L^2(\Ohi)} \leq C
  h_{K_{\bu}}^{d/2}$.
  From \eqref{eq_jumpvh} and the inverse estimate, we derive that 
  \begin{align*}
    &|w_h^{\pi_0}(\bu) - w_h^{\pi_1}(\bu)|  \leq \| w_h^{\pi_0} -
    w_h^{\pi_1} \|_{L^{\infty}(\Gamma_{K_{\bu}})} + h_{K_{\bu}} \|
    \nabla (w_h^{\pi_0}(\bu) - w_h^{\pi_1}(\bu))
    \|_{L^{\infty}({K_{\bu}})} \\
    & \leq C (h_{K_{\bu}}^{(1-d)/2} \| \jump{E_{{K_{\bu}}, \tau}^m w_h
    } \|_{L^2(\Gamma_{\Bt{{K_{\bu}}}})} + h_{K_{\bu}}^{1 - d/2}
    \|\nabla w_h^{\pi_0} \|_{L^2({K_{\bu}})} + h_{K_{\bu}}^{1 - d/2}
    \|\nabla w_h^{\pi_1} \|_{L^2({K_{\bu}})}).
  \end{align*}
  Summation over all nodes in $\overline{\OhG}$ leads to
  \begin{displaymath}
    \|  w_{h, \bc}^{\pi_i} - w_h^{\pi_i} \|_{L^2(\Ohi)} \leq C h
    (\shnorm{w_h} + \Vhnorm{w_h}),
  \end{displaymath}
  which brings the estimate \eqref{eq_Vhdiff} and completes the proof.
\end{proof}
As can be seen from \eqref{eq_newhatw}, the newly constructed function
$w_{h, \bc}$ in Lemma \ref{le_Vhdiff} depends linearly on the given
function $w_h$.
This fact enables us to define
a linear operator $\mE_h: V_h^m \rightarrow V_{h,
\bc}^m$ by setting $\mE_h w_h := w_{h, \bc}$ while $w_h - \mE_h w_h$
satisfies the estimate \eqref{eq_Vhdiff}.

We end this section by giving the $H^1$ trace estimate
\cite{Hansbo2002unfittedFEM, Guzman2018infsup} for parts of $\Gamma$. 
\begin{lemma}
  There exists a constant $h_2$ such that for any $h < h_2$, there
  holds 
  \begin{equation}
    \| w \|_{L^2(\Gamma_K)}^2 \leq C ( h_K \| \nabla w \|_{L^2(K)}^2 +
    h_K^{-1} \| w \|_{L^2(K)}^2), \quad \forall w \in H^1(K), \quad
    \forall K \in \MThG.
    \label{eq_H1trace}
  \end{equation}
  \label{le_H1trace}
\end{lemma}


\section{Numerical Scheme}
\label{sec_scheme}
In this section, we present the scheme for numerically solving the
interface problem \eqref{eq_problem}. 
The discrete variational problem is defined as: seek $u_h \in
V_h^m$ such that 
\begin{equation}
  a_h(u_h, v_h) =  l_h(v_h), \quad \forall v_h \in V_h^m.
  \label{eq_discreteweak}
\end{equation}
The bilinear form $a_h(\cdot, \cdot)$ is defined as 
\begin{displaymath}
  a_h(v_h, w_h) := b_h(v_h, w_h) + g_h(v_h, w_h) + \mu s_h(v_h, w_h),
  \quad \forall v_h, w_h \in V_h^m,
\end{displaymath}
with $\mu$ being the penalty parameter, and the forms are defined as
follows:
\begin{align*}
  b_h(v_h, w_h) := &\int_{\Omega_0 \cup \Omega_1}  \alpha \nabla v_h
  \cdot \nabla w_h \dbx{x} \\
   - &\int_{\Gamma} \aver{\alpha \nabla v_h}
  \cdot \jump{w_h} \dbx{s}  - \int_{\Gamma} \aver{ \alpha \nabla w_h}
  \cdot \jump{v_h} \dbx{s}, \quad \forall v_h, w_h \in V_h^m,
\end{align*}
\begin{displaymath}
  g_h(v_h, w_h) :=
  g_{h, 0}(v_h^{\pi_0}, w_h^{\pi_0}) + 
  g_{h, 1}(v_h^{\pi_1}, w_h^{\pi_1}), \quad \forall v_h, w_h \in
  V_h^m,
\end{displaymath}
and the form $s_h(\cdot, \cdot)$ is given by \eqref{eq_shform}. 
The linear form $l_h(\cdot)$ is defined as 
\begin{align*}
  l_h(v_h)  := &\int_{\Omega_0 \cup \Omega_1} f v_h \dbx{x} +
  \int_{\Gamma} g_N \aver{v_h} \dbx{s} \\ 
  &- \int_{\Gamma} \aver{\alpha \nabla v_h} \cdot \un_{\Gamma} g_D
  \dbx{s} + \sum_{K \in \MThG} \int_{\Gamma_{\Bt{K}}} g_D \un_{\Gamma}
  \cdot \jump{E_{K, \tau}^m v_h} \dbx{s} , \quad  \forall v_h \in
  V_h^m.
\end{align*}
Since $g_D \in L^2(\Gamma)$, it is well-defined on $\Gamma_{\Bt{K}}$.
\begin{remark}
  We note that for the case where the exact solution is smooth enough,
  i.e. $u \in H^2(\Omega_0 \cup \Omega_1)$, the error estimation is
  standard, see \cite{Hansbo2002unfittedFEM, Gurkan2019stabilized}. 
  In this case, the numerical error on the interface can be directly
  bounded by applying the $H^1$ trace estimate \eqref{eq_H1trace} to
  the term $\| \aver{\alpha \nabla(u - u_h)} \|_{L^2(\Gamma_K)}$. 
  For the case of low regularity, i.e. $s < 1$, the $H^1$ trace
  estimate is unavailable for numerical errors in the interface.
  \label{re_error}
\end{remark}
The error analysis follows the framework in \cite{Gudi2010new},
where the error estimate in the energy norm with the low regularity
assumption was developed for the interior penalty discontinuous
Galerkin methods. 
We introduce the seminorms and the norms as follows:
\begin{displaymath}
  \begin{aligned}
    \Vnorm{v}^2 & := \| \nabla v\|_{L^2(\Omega_0)}^2 + \| \nabla v
    \|_{L^2(\Omega_1)}^2, && \forall v \in H^1(\Omega_0 \cup
    \Omega_1), \\
    \DGenorm{v_h}^2 & := \| \nabla v_h \|_{L^2(\Omega_0)}^2 + \| \nabla
    v_h \|_{L^2(\Omega_1)}^2 + \shnorm{v_h}^2 + \ghnorm{v_h}^2, 
    && \forall v_h \in V_h^m, \\
    \ghnorm{v_h}^2 & := \ghonorm{v_h^{\pi_0}}^2 +
    \ghlnorm{v_h^{\pi_1}}^2, 
    && \forall v_h \in V_h^m.
  \end{aligned}
\end{displaymath}
The seminorm \eqref{eq_Vhnorm} will also be used in the error
estimation. 
By \eqref{eq_ghP1}, we immediately find that 
$\Vhnorm{v_h} \leq C \DGenorm{v_h}$ for any $v_h \in V_h^m$.
We further define a quantity $\sgDnorm{\cdot}$ as 
\begin{displaymath}
  \sgDnorm{v_h}^2 := \sum_{K \in \MThG} h_K^{-1}
  \| g_D \un_{\Gamma} - \jump{E_{K, \tau}^m
  v_h} \|^2_{L^2(\Gamma_{\Bt{K}})}, \quad \forall v_h \in V_h^m.
\end{displaymath}
We first present the coercivity to the bilinear form $a_h(\cdot,
\cdot)$.
\begin{lemma}
  Let $a_h(\cdot, \cdot)$ be defined with a sufficiently large $\mu$,
  then there holds
  \begin{equation}
    a_h(v_h, v_h) \geq C \DGenorm{v_h}^2, \quad \forall v_h \in V_h^m.
    \label{eq_coercivity}
  \end{equation}
  \label{le_coercivity}
\end{lemma}
\begin{proof}
  By the trace estimate \eqref{eq_H1trace} and the Cauchy-Schwarz
  inequality, we deduce that 
  \begin{align*}
    -2 \int_{\Gamma} &\aver{\nabla v_h} \cdot \jump{v_h} \dbx{s}  =
    \sum_{K \in \MThG} \int_{\Gamma_K} -2 \aver{\nabla v_h} \cdot
    \jump{v_h} \dbx{s} \\
    & \geq \sum_{K \in \MThG} ( - \theta h_K^{-1} \|
    \jump{v_h} \|_{L^2(\Gamma_K)}^2 - \theta^{-1} h_K \| \aver{\nabla
    v_h} \|_{L^2(\Gamma_K)}^2) \\
    & \geq -\theta \shnorm{v_h}^2 - \theta^{-1} C  (\| \nabla
    v_h^{\pi_0} \|_{L^2(\Oho)}^2 + \| \nabla
    v_h^{\pi_1} \|_{L^2(\Ohl)}^2) \geq  -\theta \shnorm{v_h}^2 -
    \theta^{-1} C \DGenorm{v_h}^2,
  \end{align*} 
  for any $\theta > 0$.
  We find that 
  \begin{align*}
    a_h(v_h, v_h) \geq & (C_0 - C_1 \theta^{-1}) \DGenorm{v_h}^2 +
    (\mu - \theta) \shnorm{v_h}^2. 
  \end{align*} 
  By selecting proper $\theta$ and $\mu$, the coercivity
  \eqref{eq_coercivity} is reached, which completes the proof. 
\end{proof}
We next prove a bound between the bilinear form $a_h(\cdot, \cdot)$
and the linear form $l_h(\cdot)$.
\begin{lemma}
  Let $u \in H^{1+s}(\Omega_0 \cup \Omega_1)$ be the solution to
  \eqref{eq_problem}, then there holds 
  \begin{equation}
    \begin{aligned}
      l_h(w_h) - a_h(v_h, w_h) \leq 
      C(\Vnorm{u - v_h} + \ghnorm{v_h} + \sgDnorm{v_h}&)
      \| \nabla w_h \|_{L^2(\Omega)}, \\ 
      &\forall (v_h, w_h) \in V_h^m \times V_{h, \bc}^m.
    \end{aligned}
    \label{eq_consist}
  \end{equation}
  \label{le_consist}
\end{lemma}
\begin{proof}
  Since $w_h \in V_{h, \bc}^m$, we know that $\jump{E_{K, \tau}^m
  w_h}|_{\Gamma_{\Bt{K}}} = \bm{0}$ for any $K \in \MThG$, which
  immediately implies $s_h(v_h, w_h) = 0$.
  From $w_h \in H_0^1(\Omega)$, we find that
  \begin{align*}
    l_h(w_h) & = l(w_h) - \int_{\Gamma} \aver{\alpha \nabla w_h} \cdot
    \un_{\Gamma} g_D
    \dbx{s} = a(u, w_h) - \int_{\Gamma} \aver{\alpha \nabla v_h} \cdot
    \un_{\Gamma}  g_D \dbx{s}.
  \end{align*}
  By the direct calculation, we have that
  \begin{align*}
    l_h(w_h) - a_h(v_h, w_h) = 
    & \int_{\Omega_0 \cup \Omega_1} \alpha
    \nabla (u - v_h) \cdot \nabla w_h \dbx{x}\\
    + &\int_{\Gamma} \aver{\alpha \nabla w_h} \cdot (\jump{v_h} - g_D
    \un_{\Gamma} ) \dbx{s} - g_h(v_h, w_h).
  \end{align*}
  Applying the Cauchy-Schwarz inequality brings 
  \begin{align*}
    \int_{\Omega_0 \cup \Omega_1} \alpha \nabla (u - v_h) \cdot \nabla
    w_h \dbx{x} - g_h(v_h, w_h) & \leq C \Vnorm{u - v_h} \Vnorm{w_h}
    + \ghnorm{v_h}\ghnorm{w_h}  \\
    & \leq C (\Vnorm{u - v_h} + \ghnorm{v_h})
    \| \nabla w_h \|_{L^2(\Omega)}.
  \end{align*}
  From the trace estimate \eqref{eq_H1trace}, we deduce that 
  \begin{align*}
    &\int_{\Gamma} \aver{\alpha \nabla w_h} \cdot (\jump{v_h} - g_D
    \un_{\Gamma})
    \dbx{s}  = \sum_{K \in \MThG} \int_{\Gamma_K}  \aver{\alpha
    \nabla w_h} \cdot (\jump{v_h} - g_D \un_{\Gamma}) \dbx{s} \\
    \leq &C \Big( \sum_{K \in \MThG}  h_K \| \aver{\nabla w_h}
    \|_{L^2(\Gamma_K)}^2 \Big)^{1/2} \Big( \sum_{K \in \MThG}
    h_K^{-1}  \| \jump{v_h} - g_D \un_{\Gamma} \|_{L^2(\Gamma_K)}^2
    \Big)^{1/2} \\
    \leq &C \Big( \sum_{K \in \MThG} \| \nabla w_h^{\pi_0}
    \|_{L^2(K)}^2 + \| \nabla w_h^{\pi_1} \|_{L^2(K)}^2
    \Big)^{1/2}\Big( \sum_{K \in \MThG} h_K^{-1} \| \jump{E_{
    K, \tau}^m v_h} - g_D \un_{\Gamma} \|_{L^2(\Gamma_{\Bt{K}})}^2
    \Big)^{1/2} \\ \leq &C \| \nabla w_h \|_{L^2(\Omega)}
    \sgDnorm{v_h}.
  \end{align*}
  Combining all above estimates leads to \eqref{eq_consist}, which
  completes the proof.
\end{proof}
We are ready to give the error estimate for the numerical solution, by
combining Lemma \ref{le_Vhdiff}, Lemma \ref{le_coercivity} and Lemma
\ref{le_consist}.
\begin{theorem}
  Let $a_h(\cdot, \cdot)$ be defined with sufficiently large $\mu$,
  and let $u_h \in V_h^m$ be the numerical solution of
  \eqref{eq_discreteweak}, and let $u \in H^{1+s}(\Omega_0 \cup
  \Omega_1)$ be the exact solution of \eqref{eq_problem}, then there
  holds
  \begin{equation}
    \DGenorm{u - u_h} \leq C h^t \| u \|_{H^{s+1}(\Omega_0 \cup
    \Omega_1)}, \quad t = \min(s, m).
    \label{eq_DGerror}
  \end{equation}
  \label{th_DGerror}
\end{theorem}
\begin{proof}
  For $s < 1$, we let $v_{h, i}$ be the linear Scott-Zhang interpolant
  of the extension $u^i$ into the space $V_{h, i}^m$, while for $s
  \geq 1$ we let $v_{h, i}$ be the Scott-Zhang interpolant of degree
  $m$ into the space $V_{h, i}^m$.
  We set $v_h := v_{h, 0} \cdot \chi_0 + v_{h, 1} \cdot \chi_1$ as the
  interpolant of $u$ into the space $V_h^m$.
  From the approximation property of $v_{h, i}$, there holds
  \begin{equation}
    \begin{aligned}
      h \| u - &v_h \|_{L^2(\Omega)} + \Vnorm{u - v_h}  \\
      & \leq \sum_{i = 0, 1} \big( h \| u^i - v_{h, i} \|_{L^2(\Ohi)}
      + |\nabla(u^i - v_{h, i}) |_{L^2(\Ohi)} \big) \leq C h^t \| u
      \|_{H^{s+1}(\Omega_0 \cup \Omega_1)}.
    \end{aligned}
    \label{eq_SZapp}
  \end{equation}
  We first present the approximation estimates of $v_h$ under the
  error measurements. 
  The property \eqref{eq_ghP2} directly implies
  \begin{equation}
    \ghnorm{v_h} \leq C h^t \| u \|_{H^{s+1}(\Omega_0 \cup \Omega_1)}.
    \label{eq_ghvhbound}
  \end{equation}
  For any cut $K \in \MThG$, we let $v_{\Bt{K}}^i \in
  \mb{P}_m(\Bt{K})$ satisfying $\|u^i - v_{\Bt{K}}^i \|_{L^2(\Bt{K})}
  \leq C h^t \| u^i \|_{H^{s+1}(\Bt{K})}$. 
  From \eqref{eq_EKL2} and the trace estimate \eqref{eq_H1trace},
  we obtain 
  \begin{align*}
    h_K^{-1/2} &\| E_{K, \tau}^m v_{h, i}  - u^i
    \|_{L^2(\Gamma_{\Bt{K}})}  \\
    & \leq h_K^{-1/2} ( \| E_{K, \tau}^m v_{h,
    i} - v_{\Bt{K}}^i 
    \|_{L^2(\Gamma_{\Bt{K}})} + \| u^i -  v_{\Bt{K}}^i 
    \|_{L^2(\Gamma_{\Bt{K}})}) \\
    &\leq C h_K^{-1/2} \| v_{h, i} - v_{\Bt{K}}^i
    \|_{L^2(K)} + C h^t \|u^i \|_{H^{s+1}(\Bt{K})} \\
    &\leq C h_K^{-1/2} ( \| v_{h, i} - u^i \|_{L^2(K)} + \|u^i -
    v_{\Bt{K}}^i \|_{L^2(K)} )+ Ch^t \|u^i \|_{H^{s+1}(\Bt{K})}
    \\
    & \leq C h^t \|u^i \|_{H^{s+1}(\Bt{K})}. 
  \end{align*}
  Summation over all cut elements brings us that 
  \begin{equation}
    \sgDnorm{v_h} \leq Ch^t \| u \|_{H^{1+s}(\Omega_0 \cup \Omega_1)}.
    \label{eq_sgDnormbound}
  \end{equation}

  We now proceed to prove the error estimate \eqref{eq_DGerror}.
  Define $w_h := u_h - v_h$. 
  From the coercivity \eqref{eq_coercivity}, we have that 
  \begin{align*}
    C\DGenorm{u_h - v_h}^2 & \leq a_h(u_h - v_h, w_h) =
    l_h(w_h) - a_h(v_h, w_h) \\
    & = l_h(\mE_h w_h) + l_h(w_h - \mE_h w_h) -
    a_h(v_h, \mE_h w_h) - a_h(v_h, w_h - \mE_h w_h).
  \end{align*}
  By the estimate \eqref{eq_Vhdiff}, we know that $\|\nabla \mE_h
  w_h\|_{L^2(\Omega)} \leq C \DGenorm{w_h}$.
  Combining \eqref{eq_consist} and the estimates \eqref{eq_SZapp} -
  \eqref{eq_sgDnormbound} yields
  \begin{align*}
    l_h(\mE_h w_h) - a(v_h, \mE_h w_h) & \leq C (\Vnorm{u - v_h} +
    \ghnorm{v_h} + \sgDnorm{v_h}) \| \nabla \mE_h w_h
    \|_{L^2(\Omega)} \\
    & \leq C h^t \|u \|_{H^{s+1}(\Omega_0 \cup \Omega_1)}
    \DGenorm{w_h}.
  \end{align*}
  Let $z_h := w_h - \mE_h w_h$, it remains to bound the term
  $l_h(z_h) - a_h(v_h, z_h)$.
  Again by \eqref{eq_Vhdiff}, it follows that
  \begin{equation}
    \| z_h^{\pi_0} \|_{L^2(\Oho)} + h \| \nabla z_h^{\pi_0}
    \|_{L^2(\Oho)} + 
    \| z_h^{\pi_1} \|_{L^2(\Ohl)} + h \| \nabla z_h^{\pi_1}
    \|_{L^2(\Ohl)} \leq C h \DGenorm{w_h}. 
    \label{eq_appzh}
  \end{equation}
  We note that $\aver{z_h}|_{\Gamma} = 0$ from the construction given
  in \eqref{eq_newhatw}. 
  By applying integration by parts, the term $l_h(z_h) - a_h(v_h,
  z_h)$ can be split into the following components:
  \begin{displaymath}
    l_h(z_h) - a_h(v_h, z_h) = 
    \tRomannum{1} + \tRomannum{2} + \tRomannum{3} + \tRomannum{4} +
    \tRomannum{5},
  \end{displaymath}
  where
  \begin{align*}
    \tRomannum{1} & := \left\{
    \begin{aligned}
      & (f, z_h)_{L^2(\Omega_0)} + (f, z_h)_{L^2(\Omega_1)}, &&
      s < 1, \\
      & \sum_{K \in \MTho} \int_{K^0} (f + \Delta v_h) z_h \dbx{x}
      + \sum_{K \in \MThl} \int_{K^1} (f + \Delta v_h) z_h \dbx{x}
      , && s \geq 1,
    \end{aligned}
    \right. \\
    \tRomannum{2} &= \int_{\Gamma}
    \aver{\alpha \nabla z_h} \cdot (g_D \un_{\Gamma} - \jump{v_h})
    \dbx{s}, \qquad \tRomannum{3} := g_h(v_h, z_h)  \\
    \tRomannum{4} & := \mu \sum_{K \in \MThG} \int_{\Gamma_{\Bt{K}}}
    h_K^{-1} (\jump{E_{K, \tau}^m v_h} - g_D \un_{\Gamma} )\cdot
    \jump{E_{K, \tau}^m z_h} \dbx{s}, \quad \\
    \tRomannum{5} & := \sum_{f \in \MFho} \int_{f^0} \alpha
    \jump{\nabla_{\un} v_h} z_h \dbx{s} - \sum_{f \in \MFhl}
    \int_{f^1} \alpha \jump{\nabla_{\un} v_h} z_h \dbx{s}.
  \end{align*}
  The estimates for each of the terms
  $\tRomannum{1}, \ldots, \tRomannum{5}$ are derived below.
  For $s < 1$, we have that 
  \begin{align*}
    \tRomannum{1} = (f, z_h )_{L^2(\Omega_0)} + (f,
    z_h )_{L^2(\Omega_1)} & \leq C \|f \|_{H^{s-1}(\Omega_0 \cup
    \Omega_1)} \| z_h \|_{H^{1-s}(\Omega_0 \cup \Omega_1)} \\
    & \leq Ch^{s} \|u \|_{H^{s+1}(\Omega_0 \cup \Omega_1)}
    \DGenorm{w_h},
  \end{align*}
  while for $s \geq 1$, we get 
  \begin{align*}
    \tRomannum{1} & =  \sum_{K \in \MTho} \int_{K^0} (f
    + \Delta v_h) z_h \dbx{x} +  \sum_{K \in \MThl} \int_{K^1}
    (f + \Delta v_h) z_h \dbx{x} \\
    & \leq \sum_{K \in \MTho}\| \Delta(u -
    v_h) \|_{L^2(K^0)} \| z_h\|_{L^2(K^0)} + \sum_{K \in
    \MThl}\| \Delta(u - v_h) \|_{L^2(K^1)} \| z_h\|_{L^2(K^1)}
    \\
    & \leq Ch^t \|u \|_{H^{s+1}(\Omega_0 \cup \Omega_1)}
    \DGenorm{w_h}.
  \end{align*}
  We apply the $H^1$ trace estimate \eqref{eq_H1trace} to bound the
  second term $\tRomannum{2}$, and it follows that
  \begin{align*}
    \tRomannum{2} &= \int_{\Gamma}  \jump{u - v_h} \cdot
    \aver{\alpha \nabla z_h} \dbx{s} = \sum_{K \in \MTh}
    \int_{\Gamma_K} \jump{u - v_h} \cdot \aver{\alpha \nabla z_h}
    \dbx{s} \\
    & \leq C \big(  \sum_{K \in \MTh} h_K^{-1} \|
    \jump{u - v_h}\|_{L^2(\Gamma_K)}^2 \big)^{1/2} \big( \sum_{K \in
    \MTh} h_K \| \aver{\alpha \nabla z_h} \|_{L^2(\Gamma_K)}^2
    \big)^{1/2} \\
    & \leq C h^t \| u \|_{H^{s+1}(\Omega_0 \cup \Omega_1)}( \|\nabla
    z_h^{\pi_0} \|_{L^2(\Oho)} +  \|\nabla z_h^{\pi_1} \|_{L^2(\Ohl)}) 
    \leq  C h^t \| u \|_{H^{s+1}(\Omega_0 \cup \Omega_1)}
    \DGenorm{w_h}. 
  \end{align*}
  By \eqref{eq_ghP1} and \eqref{eq_ghvhbound}, we observe that 
  \begin{align*}
    \tRomannum{3} \leq \ghnorm{v_h} \ghnorm{z_h} & \leq C h^t \| u
    \|_{H^{s+1}(\Omega_0 \cup \Omega_1)} ( 
    \|\nabla z_h^{\pi_0} \|_{L^2(\Oho)} +  
    \|\nabla z_h^{\pi_1} \|_{L^2(\Ohl)}) \\
    & \leq C h^t \| u \|_{H^{s+1}(\Omega_0 \cup \Omega_1)}
    \DGenorm{w_h}. 
  \end{align*}
  From \eqref{eq_sgDnormbound} and \eqref{eq_EKL2}, we have that 
  \begin{align*}
    \tRomannum{4} & \leq C \sgDnorm{v_h} \big( \sum_{K \in \MThG}
    h_K^{-1} \| \jump{ E_{K, \tau}^m z_h} \|_{L^2(\Gamma_{\Bt{K}})}^2
    \big)^{1/2} \\
    & \leq  C \sgDnorm{v_h} ( h^{-2} \| z_h^{\pi_0}
    \|_{L^2(\Oho)}^2  +  h^{-2} \| z_h^{\pi_1}
    \|_{L^2(\Ohl)}^2)^{1/2} \\
    & \leq C h^t \| u \|_{H^{s+1}(\Omega_0 \cup \Omega_1)}
    \DGenorm{w_h}. 
  \end{align*} 
  To bound the last term $\tRomannum{5}$, for any $f \in \MFhi$ we let
  $K_{f, -}, K_{f, +} \in \MThi$ be the two adjacent elements sharing
  the common face $f$. 
  Let $S_f := \text{Int}(\overline{K}_{f, -} \cup \overline{K}_{f,
  +})$, and then there exists $p_f \in \mb{P}_m(S_f)$ such that 
  $\|\nabla (u^i - p_f) \|_{L^2(S_f)} \leq C h^t \| u^i
  \|_{H^{s+1}(S_f)}$.
  Let $v_{K, \pm} := v_h|_{K_{f, \pm}}$, and we further extend
  $v_{K, \pm}$ to $S_f$ by the canonical extension. 
  We deduce that 
  \begin{align*}
    &\int_{f^i} \alpha \jump{\nabla_{\un} v_h} z_h \dbx{s}  =
    \int_{f^i} \alpha \jump{\nabla_{\un} (v_h - p_f)} z_h \dbx{s} \\
    & \leq C h_f^{-1} ( \| \nabla (v_{K, -}  - p_f) \|_{L^2(K_{f,
    -})}^2 + \| \nabla (v_{K, +}  - p_f) \|_{L^2(K_{f,
    +})}^2)^{1/2} \| z_h \|_{L^2(S_f)} \\
    & \leq C h_f^{-1} (  \| \nabla (v_{K, -}  - u^i) \|_{L^2(K_{f,
    -})}^2 +  \| \nabla (v_{K, +}  - u^i) \|_{L^2(K_{f,
    +})}^2 +  \| \nabla (u^i - p_f) \|_{L^2(S_f)}^2 )^{1/2}  \| z_h
    \|_{L^2(S_f)} \\
    & \leq Ch^{t-1} \| u^i \|_{H^{s+1}(S_f)} \|z_h  \|_{L^2(S_f)}. 
  \end{align*}
  Then, there holds 
  \begin{align*}
    \tRomannum{5} & \leq Ch^{t}   \| u \|_{H^{s+1}(\Omega_0 \cup
    \Omega_1)} ( h^{-1} \| z_h^{\pi_0} \|_{L^2(\Oho)}  +  h^{-1} \|
    z_h^{\pi_1} \|_{L^2(\Ohl)})\\
    &\leq  Ch^{t}   \| u
    \|_{H^{s+1}(\Omega_0 \cup \Omega_1)} \DGenorm{w_h}.
  \end{align*}
  Combining all estimates of $\tRomannum{1}, \ldots, \tRomannum{5}$ 
  yields  
  \begin{equation}
    C \DGenorm{u_h - v_h} \leq C h^t \| u \|_{H^{s+1}(\Omega_0 \cup
    \Omega_1)},
    \label{eq_enormuhvh}
  \end{equation}
  and the desired error estimate \eqref{eq_DGerror} thus follows from
  the triangle inequality, 
  which completes the proof.
\end{proof}
For any $w_h \in V_h^m$, we let $z_h := \mE_h w_h$. Since $z_h \in
H_0^1(\Omega)$,  there holds $\| z_h \|_{L^2(\Omega)} \leq C \| \nabla z_h
\|_{L^2(\Omega)}$ by the Poincar\'{e} inequality. 
For both $i = 0, 1$, we apply \eqref{eq_Vhdiff} to derive that
\begin{align*}
  \| &w_h^{\pi_i} \|_{L^2(\Ohi)}  \leq \| w_h^{\pi_i} - z_h^{\pi_i}
  \|_{L^2(\Ohi)} + \|z_h^{\pi_i} \|_{L^2(\Ohi)}\\
  & \leq \DGenorm{w_h} +
  \| z_h \|_{L^2(\Omega)}  \leq \DGenorm{w_h} + C \| \nabla z_h
  \|_{L^2(\Omega)} \\
  & \leq C (\DGenorm{w_h} + \| \nabla (z_h^{\pi_0} - w_h^{\pi_0})
  \|_{L^2(\Oho)} + \| \nabla (z_h^{\pi_1} - w_h^{\pi_1})
  \|_{L^2(\Ohl)} ) \leq C \DGenorm{w_h}.
\end{align*}
We conclude that the energy norm $\DGenorm{\cdot}$ is stronger than
the $L^2$ norm on $V_h^m$,
\begin{equation}
  \| w_h^{\pi_0} \|_{L^2(\Oho)} + 
  \| w_h^{\pi_1} \|_{L^2(\Ohl)} \leq C \DGenorm{w_h}, \quad \forall
  w_h \in V_h^m.
  \label{eq_L2bound}
\end{equation}
Combining \eqref{eq_L2bound} with \eqref{eq_enormuhvh}, it can be
observed that the convergence rate under the $L^2$ norm is at least
suboptimal, i.e.
\begin{equation}
  \|u - u_h\|_{L^2(\Omega_0 \cup \Omega_1)} \leq C h^t \| u
  \|_{H^{s+1}(\Omega_0 \cup \Omega_1)}.
  \label{eq_L2error}
\end{equation}

Ultimately, we present an estimate of the condition number for the
sparse matrix arising from the discrete system \eqref{eq_discreteweak},
which is desirable in unfitted methods.
\begin{theorem}
  Let $a_h(\cdot, \cdot)$ be defined with a sufficiently large $\mu$, 
  and let $A$ be the matrix arising from $a_h(\cdot, \cdot)$, there
  holds
  \begin{equation}
    \kappa(A) \leq C h^{-2}.
    \label{eq_kappaA}
  \end{equation}
  \label{th_kappaA}
\end{theorem}
\begin{proof}
  From the Cauchy-Schwarz inequality and the inverse estimate, it can
  be observed that 
  \begin{displaymath}
    a_h(v_h, v_h) \leq C \DGenorm{v_h}^2  \leq C h^{-2} (
    \| v_h^{\pi_0} \|_{L^2(\Oho)}^2 + 
    \| v_h^{\pi_1} \|_{L^2(\Ohl)}^2), \quad \forall v_h \in V_h^m.
  \end{displaymath}
  By the coercivity \eqref{eq_coercivity} and the estimate
  \eqref{eq_L2bound}, we find that 
  \begin{displaymath}
    (
    \| v_h^{\pi_0} \|_{L^2(\Oho)}^2 + 
    \| v_h^{\pi_1} \|_{L^2(\Ohl)}^2) 
    \leq C a_h(v_h, v_h) \leq C  h^{-2} (
    \| v_h^{\pi_0} \|_{L^2(\Oho)}^2 + 
    \| v_h^{\pi_1} \|_{L^2(\Ohl)}^2),
  \end{displaymath}
  for any $v_h \in V_h^m$.
  From the property of the $C^0$ space, the above estimate immediately
  brings us that $\kappa(A) \leq C h^{-2}$, which completes the proof.
\end{proof}


\section{Numerical Results}
\label{sec_numericalresults}
In this section, we provide a series of numerical tests to demonstrate
the numerical performance of the proposed method. 
In the following tests, the source function $f$ and the jump
conditions $g_N$, $g_D$ are chosen accordingly.
In Example 1 - Example 4, we solve the elliptic interface problem in
two dimensions defined in the squared domain $\Omega = (-1, 1)^2$. 
In Example 1 and Example 2, the interface $\Gamma$ is of class $C^2$
and is described by a level set function. 
In Example 3 and Example 4, the interface is taken to be the boundary
of an L-shaped domain, see Fig.~\ref{fig_2dinterface}.
In Example 5, we solve a three-dimensional interface problem in the
cubic domain $\Omega = (-1, 1)^3$ to illustrate the numerical
performance.

\begin{figure}[htp]
  \centering
  \begin{minipage}[t]{0.23\textwidth}
    \centering
    \begin{tikzpicture}[scale=1.39]
      \centering
      \input{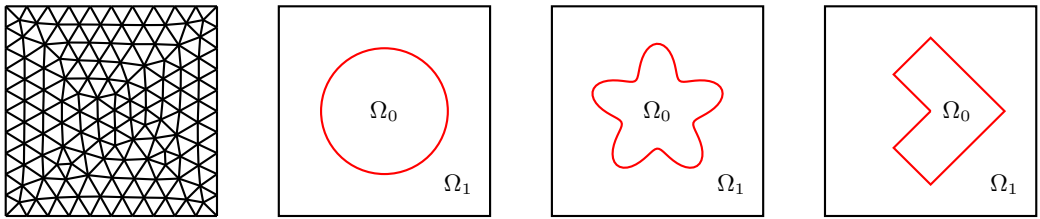}
    \end{tikzpicture}
  \end{minipage}
  \hfill
  \begin{minipage}[t]{0.23\textwidth}
    \centering
    \begin{tikzpicture}[scale=1.39]
      \centering
      \node at (0, 0) {\small $\Omega_0$};
      \node at (0.7, -0.7) {\small $\Omega_1$};
      \draw[thick, black] (-1.0, -1.0) rectangle (1.0, 1.0);
      \draw[thick, red] (0, 0)  circle [radius=0.6];
    \end{tikzpicture}
  \end{minipage}
  \hfill
  \begin{minipage}[t]{0.23\textwidth}
    \centering
    \begin{tikzpicture}[scale=1.39]
      \centering
      \node at (0, 0) {\small $\Omega_0$};
      \node at (0.7, -0.7) {\small $\Omega_1$};
      \draw[thick, black] (-1.0, -1.0) rectangle (1.0, 1.0);
      \draw[thick, domain=0:360, red, samples=120] plot (\x:{(0.5 +
      sin(\x*5)/7)*1.00});
    \end{tikzpicture}
  \end{minipage}
  \hfill
  \begin{minipage}[t]{0.23\textwidth}
    \centering
    \begin{tikzpicture}[scale=1.39]
      \centering
      \node at (0.25, 0) {\small$\Omega_0$};
      \node at (0.7, -0.7) {\small $\Omega_1$};
      \draw[thick, black] (-1.0, -1.0) rectangle (1.0, 1.0);
      \draw[thick, red] (0, 0) -- (-0.35, 0.35) -- (0, 0.7) -- (0.7,
      0.0) -- (0, -0.7) -- (-0.35, -0.35) -- (0, 0);
    \end{tikzpicture}
  \end{minipage}
  \caption{The unfitted mesh and the interface in two dimensions.}
  \label{fig_2dinterface}
\end{figure}

\begin{figure}[htp]
  \centering
  \includegraphics[width=0.29\textwidth, height=0.29\textwidth]{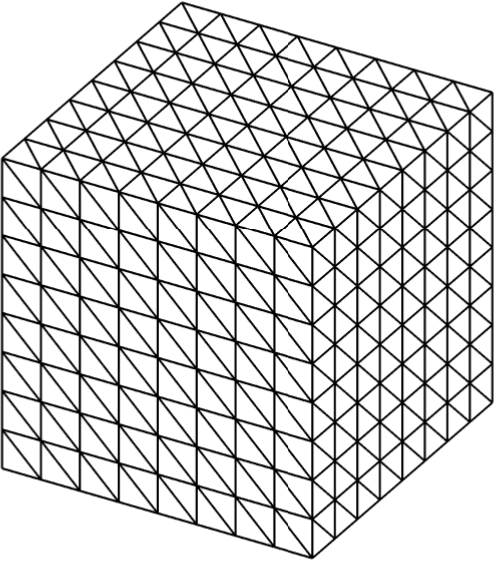}
  \hspace{100pt}
  \begin{tikzpicture}[scale=2.6]
    \draw[thick] (-0.9, 0.2) -- (0, 0) -- (0.6, 0.35);
    \draw[thick] (0, 0) -- (0, 1);
    \draw[thick] (0.6, 1.35) -- (0, 1) -- (-0.9, 1.2);
    \draw[thick] (0.6, 1.35) -- (0.6, 0.35);
    \draw[thick] (-0.9, 0.2) -- (-0.9, 1.2);
    \draw[thick] (-0.9, 1.2) -- (-0.26, 1.55) -- (0.6, 1.35);
    \draw[thick, dashed] (-0.9, 0.2) -- (-0.26, 0.55) -- (0.6, 0.35);
    \draw[thick, dashed] (-0.26, 0.55) -- (-0.26, 1.55);
    \draw[thick] (-0.2, 0.7) circle [radius=0.3];
    \draw[thick, dashed] (-0.5, 0.7) [out = -30, in = 210] to (0.1,
    0.7);
    \draw[thick, dashed] (-0.5, 0.7) [out = 20, in = 160] to (0.1,
    0.7);
    \node at (-0.16, 0.68) {$\Omega_0$};
    \node at (0.3, 0.89) {$\Omega_1$};
  \end{tikzpicture}
  \caption{The unfitted mesh and the interface in three dimensions.}
  \label{fig_3dinterface}
\end{figure}
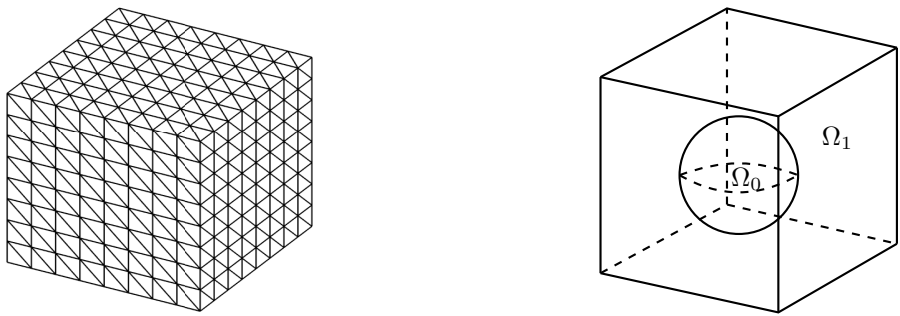

\paragraph{\textbf{Example 1.}}
We first consider an elliptic problem in two dimensions with
a circular interface centered at $(0, 0)$ with radius $r = 0.6$.
We select the exact solution as
\begin{displaymath}
  u(x, y) = \sin(\pi x) \sin(2\pi y), \quad \forall (x, y) \in
  \Omega,
\end{displaymath}
with a discontinuous parameter $\alpha|_{\Omega_0} = 1,
\alpha|_{\Omega_1} = 5$. 
We adopt a family of triangular meshes with the mesh size $h = 1/10,
\ldots, 1/80$ for this test. 
The convergence histories under the $H^1$ norm and the $L^2$ norm 
are gathered in Tab.~\ref{tab_ex1}.
For the smooth case, it can be observed that both numerical errors
approach zero at optimal rates $O(h^m)$ and $O(h^{m+1})$,
respectively, which are consistent with the theoretical analysis
derived in Theorem \ref{th_DGerror}.

In Tab.~\ref{tab_ex1cond}, we report the condition number
of the final linear system corresponding to the numerical scheme
\eqref{eq_discreteweak}. 
It is evident that the condition number $\kappa(A_m)$ scales
as $O(h^{-2})$ for all accuracies $1 \leq m \leq 3$, which confirms
the predication in Theorem \ref{th_kappaA}.

\begin{table}
  \centering
  \renewcommand\arraystretch{1.5}
  \scalebox{1.00}{
  \begin{tabular}{p{0.5cm}|p{3.8cm}| p{1.6cm} | p{1.6cm} | p{1.6cm} |
    p{1.6cm} | p{0.9cm} }
    \hline\hline
    $m$ & $h$ & 1/10 & 1/20 & 1/40 & 1/80 & rate \\
    \hline
    \multirow{2}{*}{$1$} &$ \|\nabla (u - u_h) \|_{L^2(\Omega_0 \cup
    \Omega_1)}$ &  
    1.396e+00 & 6.593e-01 & 3.213e-01 & 1.583e-01 & 1.02 \\
    \cline{2-7}
    & $\| u - u_h \|_{L^2(\Omega_0 \cup \Omega_1)}$ & 
    7.017e-02 & 1.428e-02 & 2.981e-03 & 6.399e-04 & 2.21 \\
    \hline
    \multirow{2}{*}{$2$} &$ \| \nabla(u - u_h) \|_{L^2(\Omega_0 \cup
    \Omega_1)} $ &  
    1.106e-01 & 2.599e-02 & 6.315e-03 & 1.551e-03 & 2.03 \\
    \cline{2-7}
    & $\| u - u_h \|_{L^2(\Omega_0 \cup \Omega_1)}$ & 
    1.500e-03 & 1.753e-04 & 2.089e-05 & 2.557e-06 & 3.03 \\
    \hline
    \multirow{2}{*}{$3$} & $ \| \nabla( u - u_h) \|_{L^2(\Omega_0 \cup
    \Omega_1)} $ & 
    6.406e-03 & 7.021e-04 & 8.189e-05 & 9.871e-06 & 3.05 \\
    \cline{2-7}
    & $\| u - u_h \|_{L^2(\Omega_0 \cup \Omega_1)}$ & 
    5.918e-05 & 3.283e-06 & 1.900e-07 & 1.123e-08 & 4.08 \\
    \hline\hline
  \end{tabular}
  }
  \caption{The convergence histories for Example 1.}
  \label{tab_ex1}
\end{table}

\begin{table}
  \centering
  \renewcommand\arraystretch{1.5}
  \begin{tabular}{p{0.5cm}|p{1.6cm}|p{1.6cm}| p{1.6cm}|p{1.6cm}|p{1cm}}
    \hline\hline
    $m$ & 1/10 & 1/20 & 1/40 & 1/80 & rate \\
    \hline
    $1$ & 1.696e+3 & 6.672e+3 & 2.755e+4 & 1.093e+5 & 
    \multirow{3}{*}{$O(h^{-2})$} \\
    \cline{1-5}
    $2$ & 3.577e+4 & 1.281e+5 & 5.302e+5 & 2.139e+6 & \\
    \cline{1-5}
    $3$ & 2.203e+6 &  4.137e+6 & 1.719e+7 & 6.832e+7 & \\
    \hline\hline
  \end{tabular}
  \caption{The condition numbers for the final sparse linear system
  for Example 1.}
  \label{tab_ex1cond}
\end{table}

\paragraph{\textbf{Example 2.}}
In the second test, we solve an interface problem which involves 
an interface consisting of both concave and convex curve segments 
\cite{Zhou2006fictitious}, see Fig.~\ref{fig_2dinterface}. 
The star-shaped interface is parametrized with polar coordinates $(r,
\theta)$, that is
\begin{displaymath}
  r = \frac{1}{2} + \frac{\sin 5 \theta}{7}.
\end{displaymath}
The exact solution is chosen to be discontinuous across the interface, 
which reads
\begin{displaymath}
  u(x, y) = \begin{cases}
    \sin(2\pi x) \sin(4 \pi y), & \text{in } \Omega_0, \\
    \cos(2\pi x) \cos(4 \pi y), & \text{in } \Omega_1, \\
  \end{cases}
  \quad
  \alpha = \begin{cases}
    10, & \text{in } \Omega_0, \\
    1, & \text{in } \Omega_1. \\
  \end{cases}
\end{displaymath}
The numerical errors under both error measurements are collected in
Tab.~\ref{tab_ex2}. 
The detected convergence orders are $O(h^{m})$ and
$O(h^{m+1})$ for the errors $\| \nabla (u - u_h) \|_{L^2(\Omega_0 \cup
\Omega_1)}$ and $ \| u - u_h \|_{L^2(\Omega_0 \cup \Omega_1)} $,
respectively. 
As in the previous example, the convergence rates are both optimal,
and the convergence histories also agree with the theoretical results
for discontinuous exact solutions.

\begin{table}
  \centering
  \renewcommand\arraystretch{1.5}
  \scalebox{1.00}{
  \begin{tabular}{p{0.5cm}|p{3.8cm}| p{1.6cm} | p{1.6cm} | p{1.6cm} |
    p{1.6cm} | p{0.9cm} }
    \hline\hline
    $m$ & $h$ & 1/10 & 1/20 & 1/40 & 1/80 & rate \\
    \hline
    \multirow{2}{*}{$1$} & $ \|\nabla(u - u_h) \|_{L^2(\Omega_0 \cup
    \Omega_1)} $ &  
    2.573e+00 & 1.266e+00 & 6.256e-01 & 3.117e-01 & 1.01 \\
    \cline{2-7}
    & $\| u - u_h \|_{L^2(\Omega_0 \cup \Omega_1)}$ & 
    5.509e-02 & 1.183e-02 & 2.550e-03 & 5.866e-04 & 2.12 \\
    \hline
    \multirow{2}{*}{$2$} & $ \| \nabla(u -u_h) \|_{L^2(\Omega_0 \cup
    \Omega_1)}$  &  
    3.168e-01 & 6.871e-02 & 1.496e-02 & 3.428e-03 & 2.13 \\
    \cline{2-7}
    & $\| u - u_h \|_{L^2(\Omega_0 \cup \Omega_1)}$ & 
    3.003e-03 & 2.756e-04 & 2.701e-05 & 2.963e-06 & 3.18 \\
    \hline
    \multirow{2}{*}{$3$} & $ \| \nabla(u - u_h) \|_{L^2(\Omega_0 \cup
    \Omega_1)} $ & 
    3.316e-02 & 3.263e-03 & 3.112e-04 & 3.235e-05 & 3.26 \\
    \cline{2-7}
    & $\| u - u_h \|_{L^2(\Omega_0 \cup \Omega_1)}$ & 
    2.132e-04 & 9.856e-06 & 4.503e-07 & 2.173e-08 & 4.37 \\
    \hline\hline
  \end{tabular}
  }
  \caption{The convergence histories for Example 2.}
  \label{tab_ex2}
\end{table}

\paragraph{\textbf{Example 3.}}
In this test, we consider the case with an L-shaped polygonal
interface, which is described by the following vertices 
\begin{equation}
  (0, 0), \quad (-0.35, 0.35), \quad (0, 0.7), \quad (0.7, 0), \quad
  (0, -0.7), \quad (-0.35, -0.35). 
  \label{eq_Lshaped}
\end{equation}
Particularly, the interface here fails to be of class $C^2$.
We first choose a smooth analytical solution defined as 
\begin{displaymath}
  u(x, y) = \begin{cases}
    \cos(\pi x)\cos(\pi y), & \text{in } \Omega_0, \\
    \sin(2\pi x)\sin(2\pi y), & \text{in } \Omega_1, \\
  \end{cases} \quad \alpha = 1, \quad \text{in } \Omega,
\end{displaymath}
to test the numerical scheme.
The numerical errors under $L^2$ and $H^1$ norms are shown in
Tab.~\ref{tab_ex3}. 
It is shown that both the $L^2$- and $H^1$-norm errors achieve optimal
convergence rates to zero, which validates the theoretical estimation 
for the case with the polygonal interface.

\begin{table}
  \centering
  \renewcommand\arraystretch{1.5}
  \scalebox{1.00}{
  \begin{tabular}{p{0.5cm}|p{3.8cm}| p{1.6cm} | p{1.6cm} | p{1.6cm} |
    p{1.6cm} | p{0.9cm} }
    \hline\hline
    $m$ & $h$ & 1/10 & 1/20 & 1/40 & 1/80 & rate \\
    \hline
    \multirow{2}{*}{$1$} & $ \|\nabla(u - u_h) \|_{L^2(\Omega_0 \cup
    \Omega_1)} $ &  
    1.532e+00 & 7.688e-01 & 3.838e-01 & 1.918e-01 & 1.00 \\
    \cline{2-7}
    & $\| u - u_h \|_{L^2(\Omega_0 \cup \Omega_1)}$ & 
    3.482e-02 & 8.836e-03 & 2.193e-03 & 5.468e-04 & 2.01 \\
    \hline
    \multirow{2}{*}{$2$} & $ \| \nabla(u -u_h) \|_{L^2(\Omega_0 \cup
    \Omega_1)}$  &  
    1.351e-01 & 3.393e-02 & 8.418e-03 & 2.096e-03 & 2.01 \\
    \cline{2-7}
    & $\| u - u_h \|_{L^2(\Omega_0 \cup \Omega_1)}$ & 
    1.633e-03 & 2.055e-04 & 2.542e-05 & 3.152e-06 & 3.02 \\
    \hline
    \multirow{2}{*}{$3$} & $ \| \nabla(u - u_h) \|_{L^2(\Omega_0 \cup
    \Omega_1)} $ & 
    7.831e-03 & 9.868e-04 & 1.192e-04 & 1.476e-05 & 3.01 \\
    \cline{2-7}
    & $\| u - u_h \|_{L^2(\Omega_0 \cup \Omega_1)}$ & 
    6.516e-05 & 4.086e-06 & 2.431e-07 & 1.493e-08 & 4.03 \\
    \hline\hline
  \end{tabular}
  }
  \caption{The convergence histories for Example 3.}
  \label{tab_ex3}
\end{table}

\paragraph{\textbf{Example 4.}}
In this test, we investigate the performance of the proposed method 
dealing with the problem that involves a singular solution. 
In the L-shaped domain \eqref{eq_Lshaped}, 
we select
\begin{displaymath}
  u(x, y) = \begin{cases}
    r^{\beta} \sin(\beta \theta), & \text{in } \Omega_0, \\ 
    0, & \text{in } \Omega_1, \\
  \end{cases} \quad \alpha = 1,
\end{displaymath}
to be the exact solution in polar coordinates $(r, \theta)$.
Here, $u$ has the regularity $H^{1 + \beta - \varepsilon}(\Omega_0
\cup \Omega_1)$ for any $\varepsilon > 0$.
We take $\beta = 2/3$, and
the numerical results are listed in Tab.~\ref{tab_ex41}.
It can be seen that the convergence rates under the $L^2$ norm and the
$H^1$ norm are nearly $O(h^{0.66})$ for all accuracies $m= 1,2,3$,
which is consistent with the regularity of the solution. 
Different from the case where the exact solution is smooth, 
the rate under $L^2$ norm is observed to be only suboptimal, and the
reason may be traced back to the singularity of the exact solution and
the polygonal interface. 
Although the convergence rates for all accuracies are the same, the
high-order scheme delivers smaller numerical errors. 

Furthermore, we solve the problem with $\beta = 1/4$, for which the
solution exhibits stronger singularity near the origin. 
The results are gathered in Tab.~\ref{tab_ex42}.
The numerical errors under both error measurements are also found to
decrease at the speed $O(h^{0.25})$, which are in agreement with the
error analysis. 
For such cases, applying adaptive mesh refinement strategies may
potentially improve the convergence performance, and this will be
explored in future studies.

\begin{table}
  \centering
  \renewcommand\arraystretch{1.5}
  \scalebox{1.00}{
  \begin{tabular}{p{0.5cm}|p{3.8cm}| p{1.6cm} | p{1.6cm} | p{1.6cm} |
    p{1.6cm} | p{0.9cm} }
    \hline\hline
    $m$ & $h$ & 1/10 & 1/20 & 1/40 & 1/80 & rate \\
    \hline
    \multirow{2}{*}{$1$} & $ \|\nabla(u - u_h) \|_{L^2(\Omega_0 \cup
    \Omega_1)} $ &  
    8.156e-02 & 5.203e-02 & 3.307e-02 & 2.096e-02 & 0.66 \\
    \cline{2-7}
    & $\| u - u_h \|_{L^2(\Omega_0 \cup \Omega_1)}$ & 
     3.539e-03 & 1.925e-03 & 1.150e-03 & 7.119e-04 & 0.69 \\
    \hline
    \multirow{2}{*}{$2$} & $ \| \nabla(u -u_h) \|_{L^2(\Omega_0 \cup
    \Omega_1)}$  &  
    3.278e-02 & 2.065e-02 & 1.301e-02 & 8.193e-03 & 0.67 \\
    \cline{2-7}
    & $\| u - u_h \|_{L^2(\Omega_0 \cup \Omega_1)}$ & 
    2.840e-03 & 1.782e-03 & 1.121e-03 & 7.063e-04 & 0.68 \\
    \hline
    \multirow{2}{*}{$3$} & $ \| \nabla(u - u_h) \|_{L^2(\Omega_0 \cup
    \Omega_1)} $ & 
    1.907e-02 & 1.214e-02 & 7.730e-03 & 4.920e-03 & 0.65 \\
    \cline{2-7}
    & $\| u - u_h \|_{L^2(\Omega_0 \cup \Omega_1)}$ & 
    2.827e-03 & 1.780e-03 & 1.121e-03 & 7.063e-04 & 0.66 \\
    \hline\hline
  \end{tabular}
  }
  \caption{The convergence histories for Example 4 with the solution
  $u \in H^{5/3 - \varepsilon}(\Omega_0 \cup \Omega_1)$.}
  \label{tab_ex41}
\end{table}

\begin{table}
  \centering
  \renewcommand\arraystretch{1.5}
  \scalebox{1.00}{
  \begin{tabular}{p{0.5cm}|p{3.8cm}| p{1.6cm} | p{1.6cm} | p{1.6cm} |
    p{1.6cm} | p{0.9cm} }
    \hline\hline
    $m$ & $h$ & 1/10 & 1/20 & 1/40 & 1/80 & rate \\
    \hline
    \multirow{2}{*}{$1$} & $ \|\nabla(u - u_h) \|_{L^2(\Omega_0 \cup
    \Omega_1)} $ &  
    2.949e-01 & 2.494e-01 & 2.108e-01 & 1.782e-01 & 0.24 \\
    \cline{2-7}
    & $\| u - u_h \|_{L^2(\Omega_0 \cup \Omega_1)}$ & 
    3.081e-02 & 2.563e-02 & 2.149e-02 & 1.806e-02 & 0.25 \\
    \hline
    \multirow{2}{*}{$2$} & $ \| \nabla(u -u_h) \|_{L^2(\Omega_0 \cup
    \Omega_1)}$  &  
    2.105e-01 & 1.789e-01 & 1.519e-01 & 1.290e-01 & 0.24 \\
    \cline{2-7}
    & $\| u - u_h \|_{L^2(\Omega_0 \cup \Omega_1)}$ & 
    3.043e-02 & 2.551e-02 & 2.147e-02 & 1.805e-02 & 0.25 \\
    \hline
    \multirow{2}{*}{$3$} & $ \| \nabla(u - u_h) \|_{L^2(\Omega_0 \cup
    \Omega_1)} $ & 
    1.695e-01 & 1.448e-01 & 1.236e-01 & 1.055e-01 & 0.23 \\
    \cline{2-7}
    & $\| u - u_h \|_{L^2(\Omega_0 \cup \Omega_1)}$ & 
    3.008e-02 & 2.503e-02 & 2.103e-02 & 1.771e-02 & 0.25 \\
    \hline\hline
  \end{tabular}
  }
  \caption{The convergence histories for Example 4 with the solution
  $u \in H^{1.25 - \varepsilon}(\Omega_0 \cup \Omega_1)$.}
  \label{tab_ex42}
\end{table}

\paragraph{\textbf{Example 5.}}
In this example, we solve a three-dimensional problem in the cubic
domain $\Omega = (-1, 1)^3$, see Fig.~\ref{fig_3dinterface}. 
The interface $\Gamma$ is taken as a sphere centered at the origin
with radius $r = 0.6$.
Let the exact solution $u$ be 
\begin{displaymath}
  u(x, y, z) = \begin{cases}
    \sin(\pi x) \sin(\pi y) \sin(\pi z), & \text{in } \Omega_0, \\ 
    \cos(\pi x) \cos(\pi y) \cos(\pi z), & \text{in } \Omega_1, \\ 
  \end{cases} \quad \alpha = 1, \quad \text{in } \Omega.
\end{displaymath}
We adopt a series of tetrahedral meshes with the mesh size $h = 1/8,
\ldots, 1/64$ to solve this problem. 
The convergence results are reported in Tab.~\ref{tab_ex5}. 
It is shown that the numerical solutions have optimal convergence
rates under both the $L^2$ and $H^1$ norms, thus illustrating the
accuracy of the method in three dimensions.

\begin{table}
  \centering
  \renewcommand\arraystretch{1.5}
  \scalebox{1.00}{
  \begin{tabular}{p{0.5cm}|p{3.8cm}| p{1.6cm} | p{1.6cm} | p{1.6cm} |
    p{1.6cm} | p{0.9cm} }
    \hline\hline
    $m$ & $h$ & 1/8 & 1/16 & 1/32 & 1/64 & rate \\
    \hline
    \multirow{2}{*}{$1$} & $ \|\nabla(u - u_h) \|_{L^2(\Omega_0 \cup
    \Omega_1)} $ &  
    1.395e+00 & 7.156e-01 & 3.531e-01 & 1.756e-01 & 1.01 \\
    \cline{2-7}
    & $\| u - u_h \|_{L^2(\Omega_0 \cup \Omega_1)}$ & 
    7.852e-02 & 2.279e-02 & 5.382e-03 & 1.243e-03 & 2.11 \\ 
    \hline
    \multirow{2}{*}{$2$} & $ \| \nabla(u -u_h) \|_{L^2(\Omega_0 \cup
    \Omega_1)}$  &  
    2.905e-01 & 6.942e-02 & 1.582e-02 & 3.753e-03 & 2.08 \\
    \cline{2-7}
    & $\| u - u_h \|_{L^2(\Omega_0 \cup \Omega_1)}$ & 
    8.336e-03 & 7.963e-04 & 7.060e-05 & 7.612e-06 & 3.21 \\
    \hline\hline
  \end{tabular}
  }
  \caption{The convergence histories for Example 5.}
  \label{tab_ex5}
\end{table}


\section{Conclusions}
\label{sec_conclusion}
In this paper, we develop an unfitted finite element method for the
elliptic interface problem. 
We consider the case where the interface is $C^2$-smooth or polygonal,
and the exact solution is assumed to be piecewise $H^{1+s}$ with any
$s > 0$.
The stability near the interface is ensured by the local polygonal
extension technique, which is further employed in the construction of
the jump penalty term. 
From this technique, 
we establish an inverse-type estimate on the interface, which enables
us to derive the error estimates of the low regularity case.
Finally, a series of numerical experiments in two and three
dimensions are conducted to validate the theoretical findings.



\begin{appendix}
  \label{sec_app}
  \section{}
  Let $T$ denote a $d$-dimensional simplex with vertices $A_0, \ldots,
  A_{d}$, and let $L$ be the length of the longest side of $T$. 
  \begin{lemma}
    For any interior point $P$ in $T$, there holds
    \begin{equation}
      \sum_{j = 0}^d |PA_j| < dL. 
      \label{eq_app_dis}
    \end{equation}
    \label{le_app_dis}
  \end{lemma}
  \begin{proof}
    We first prove \eqref{eq_app_dis} in two dimensions. 
    Assume that $A_{1}A_{2}$ is the shortest side of the triangle $A_0
    A_1A_2$, and we let $D$ and $E$ be points on sides $A_0A_1$ and
    $A_0A_2$, respectively, such that the segment $DE$ passes through
    $P$ and is parallel to the edge $A_1A_2$, see
    Fig.~\ref{fig_app_proof}. 
    Then, $DE$ is the shortest side of the
    triangle $A_0DE$, i.e. $|DE| \leq \min(|A_0D|, |A_0E|).$
    Notice that $|A_0P| < \max(|A_0 D|, |A_0 E|)$, which gives
    $|A_0 P| + |DE| \leq |A_0D| + |A_0 E|$.
    By the triangle inequalities $|A_1P| < |A_1D| + |DP|$ and $|A_2P|
    < |PE| + |A_2E|$, we conclude that 
    \begin{equation}
      |A_0P| + |A_1P| + |A_2P| < |A_0A_1| + |A_0A_2| \leq 2L,
      \label{eq_app_dis_2d}
    \end{equation}
    which brings \eqref{eq_app_dis} in two dimensions. 

    In three dimensions, we let $D$, $F$, $E$ be the points on edges
    $A_0A_1$, $A_0A_2$, $A_0A_3$, respectively, such that the triangle
    $DFE$ contains $P$ and is parallel to the triangle $A_1 A_2 A_3$,
    see Fig.~\ref{fig_app_proof}.
    Note that there exists a constant $k \in (0, 1)$ such that 
    \begin{displaymath}
      \max(|DF|, |DE|, |FE|, |A_0D|, |A_0E|, |A_0F|) \leq k
      L, \quad \max(|A_1D|, |A_2F|, |A_3E|) \leq (1-k)L.
    \end{displaymath}
    Applying \eqref{eq_app_dis_2d} to the triangle $DFE$, 
    there holds $|DP| + |EP| + |FP| < 2kL$, and $|A_0P| < kL$. 
    We deduce that 
    \begin{align*}
      |A_0P| + |A_1P| + |A_2P| + |A_3P| &< |DP| + |A_1D| + |PE| +
      |A_3E| + |PF| + |A_2F| + |A_0P| \\
      & < 2kL + 3(1-k)L + kL = 3L,
    \end{align*}
    which leads to \eqref{eq_app_dis} and completes the proof.
    \begin{figure}[htp]
      \centering
      \begin{minipage}[t]{0.42\textwidth}
        \centering
        \begin{tikzpicture}[scale=3.6]
          \coordinate (A) at (0.35, 0.75);
          \coordinate (B) at (0.0, 0.00);
          \coordinate (C) at (0.8, 0.00);
          \coordinate (P) at (0.3, 0.25);
          \coordinate (D) at (0.1166, 0.25);
          \coordinate (E) at (0.65, 0.25);
          \draw[thick] (A) -- (B) -- (C) -- (A);
          \node[above] at (A) {$A_0$};
          \node[below left] at (B) {$A_1$};
          \node[below left] at (D) {$D$};
          \node[below right] at (C) {$A_2$};
          \node[right] at (E) {$E$};
          \node[below] at (P) {$P$};
          \draw[dashed, thick] (D) -- (E);
          \draw[dashed, thick] (B) -- (P);
          \draw[dashed, thick] (A) -- (P) -- (C);
          \draw[thick, fill=black] (P) circle [radius=0.012];
          \draw[thick, fill=black] (D) circle [radius=0.012];
          \draw[thick, fill=black] (E) circle [radius=0.012];
        \end{tikzpicture}
      \end{minipage}
      \begin{minipage}[t]{0.42\textwidth}
        \centering
        \begin{tikzpicture}[scale=2]
          \coordinate (A0) at (0., 1);
          \coordinate (A1) at (-0.2, -0.35);
          \coordinate (A2) at (-0.6, 0.00);
          \coordinate (A3) at (0.8, 0.00);
          \coordinate (D) at (-0.133333, 0.1);
          \coordinate (F) at (-0.4, 0.333333);
          \coordinate (E) at (0.5333333, 0.333333);
          \coordinate (P) at (-0., 0.2555555);
          \draw[thick] (A0) -- (A1) -- (A2) -- (A0);
          \draw[thick] (A0) -- (A3) -- (A1);
          \draw[thick, dashed] (A2) -- (A3);
          \draw[thick] (F) -- (D) -- (E);
          \draw[thick, dashed] (F) -- (E);
          \node[above] at (A0) {$A_0$};
          \node[below] at (A1) {$A_1$};
          \node[left] at (A2) {$A_2$};
          \node[right] at (A3) {$A_3$};
          \node[below left] at (D) {$D$};
          \node[left] at (F) {$F$};
          \node[right] at (E) {$E$};
          \node[above right] at (P) {$P$};
          \draw[dashed, thick] (P) -- (E);
          \draw[dashed, thick] (P) -- (F);
          \draw[dashed, thick] (A0) -- (P) -- (D);
          \draw[dashed, thick] (P) -- (A3);
          \draw[dashed, thick] (P) -- (A1);
          \draw[dashed, thick] (P) -- (A2);
          \draw[thick, fill=black] (P) circle [radius=0.025];
          \draw[thick, fill=black] (F) circle [radius=0.025];
          \draw[thick, fill=black] (D) circle [radius=0.025];
          \draw[thick, fill=black] (E) circle [radius=0.025];
        \end{tikzpicture}
      \end{minipage}
      \caption{The simplex with vertices $A_0,\ldots,A_d$ in two and
      three dimensions.}
      \label{fig_app_proof}
    \end{figure}
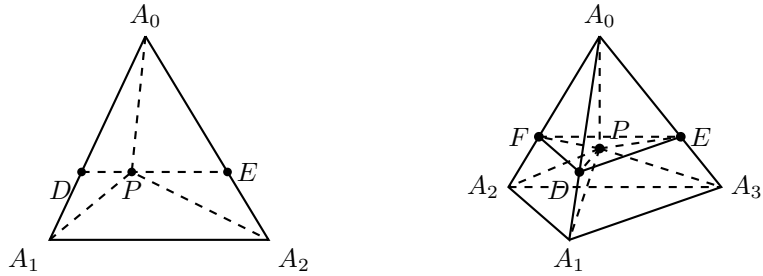
  \end{proof}
  Let $Q$ denote the barycenter of $T$. 
  From \eqref{le_app_dis}, it is straightforward to see that
  \begin{equation}
    |PQ| \leq \frac{1}{d + 1} \sum_{j = 0}^d |PA_j| <
    \frac{dL}{d + 1}.
    \label{eq_app_bc}
  \end{equation}
  for any interior point $P$ in $T$. 
\end{appendix}

\bibliographystyle{amsplain}
\bibliography{../ref}

\end{document}